\documentclass{amsart}
\usepackage{amsmath}
\usepackage{dsfont}
  \usepackage{paralist}
  \usepackage{graphics} 
  \usepackage{epsfig} 
\usepackage{graphicx}  \usepackage{epstopdf}
 \usepackage[colorlinks=true]{hyperref}
 \usepackage{amsmath}
\usepackage{amsfonts}
\usepackage{amssymb}
\usepackage{ mathrsfs }
\usepackage{amsthm}
\usepackage[pagewise]{lineno}
\usepackage{enumitem}
\usepackage{graphicx,color,xcolor,tikz}
\usepackage{bm}
\usepackage{hyperref}
\hypersetup{
    colorlinks=true,                         
    linkcolor=blue, 
    citecolor=red, 
  } 
\hypersetup{urlcolor=blue, citecolor=red}

\newtheorem{thm}{Theorem}
\newtheorem{prop}[thm]{Proposition}
\newtheorem{lem}[thm]{Lemma}

\newtheorem{defi}[thm]{Definition}
\newtheorem{rem}[thm]{Remark}

\def\I{\mathcal I}

\def\M{\mathcal M}

\def\R{\mathbb R}

\def\T{\mathcal T}

\def\e{\varepsilon}

\newcommand{\pt}{\partial}

\newcommand{\abs}[1]{\ensuremath{\left|#1\right|}}
\newcommand{\card}[1]{\left\lvert#1\right\rvert}
\newcommand{\norm}[1]{\ensuremath{\left\|#1\right\|}}

\makeatletter
\newcommand{\doublewidetilde}[1]{{%
  \mathpalette\double@widetilde{#1}%
}}
\newcommand{\double@widetilde}[2]{%
  \sbox\z@{$\m@th#1\widetilde{#2}$}%
  \ht\z@=.9\ht\z@
  \widetilde{\box\z@}%
}
\makeatother

\title[Part 2- Derivation of the macroscopic tridomain model]{Microscopic tridomain model of electrical activity in the heart with dynamical gap junctions. Part 2- Derivation of the macroscopic tridomain model by unfolding homogenization method}

\subjclass{65N55
, 35A01
, 35A02
, 35B27
, 35K57.
}
 \keywords{Tridomain model, reaction-diffusion system, homogenization theory, time-periodic unfolding method, gap junctions, cardiac electric field.}

\author{Fakhrielddine Bader$^*$ }
\address[Fakhrielddine Bader]{Institut de Recherche Mathématique de Rennes (IRMAR), UMR 6625 CNRS, Université de Rennes 1, Campus de Beaulieu, F-35042 Rennes cedex, France}
\email{fakhrielddine.bader@univ-rennes1.fr}

\author{Mostafa Bendahmane}
\address[Mostafa Bendahmane]{Institut de Mathématiques de Bordeaux (IMB) and INRIA-Carmen Bordeaux Sud-Ouest, Université de Bordeaux, 33076 Bordeaux Cedex, France}
\email{mostafa.bendahmane@u-bordeaux.fr}

\author{Mazen Saad}
\address[Mazen Saad]{Laboratoire de Mathématiques Jean Leray (LMJL), UMR 6629 CNRS, École Centrale de Nantes, 1 rue de Noé, 44321 Nantes, France}
\email{mazen.saad@ec-nantes.fr}

\author{Raafat Talhouk}
\address[Raafat Talhouk]{Léonard de Vinci Pôle Universitaire, Research Center, 92 916 Paris La Défense, France \& Department of Mathematics, Faculty of Sciences 1, Laboratory of Mathematics-DSST, Lebanese University Hadat, Lebanon}
\email{rtalhouk@ul.edu.lb}

\thanks{$^*$ Corresponding author: fakhrielddine.bader@gmail.com}

\begin{document}
\maketitle



\begin{abstract}
We study the homogenization of a novel microscopic tridomain system, allowing for a more detailed analysis of the properties of cardiac conduction than the classical bidomain and monodomain models. In \cite{BaderTridPart1}, we detail this model in which gap junctions are considered as the connections between adjacent cells in cardiac muscle and could serve as alternative or supporting pathways for cell-to-cell electrical signal propagation. Departing from this microscopic cellular model, we apply the periodic unfolding method to derive the macroscopic tridomain model.  Several difficulties prevent the application of unfolding homogenization results, including the degenerate temporal structure of the tridomain equations and a nonlinear dynamic boundary condition on the cellular membrane. To prove the convergence of the nonlinear terms, especially those defined on the microscopic interface, we use the boundary unfolding operator and a Kolmogorov-Riesz compactness's result. 
\end{abstract}
\tableofcontents

\section{Introduction}
The conduction of electrical waves in cardiac tissue is key to human life, as the synchronized contraction of the cardiac muscle is controlled by electrical impulses that travel in a coordinated manner throughout the heart chambers. Under pathological conditions cardiac conduction can be severely reduced, potentially leading to reentrant arrhythmias and ultimately death if normal propagation is not restored properly. At a sub-cellular level, electrical communication in cardiac tissue occurs by means of a rapid flow of ions moving through the cytoplasm of cardiac cells, and a slower inter-cellular flow mediated by gap junctions embedded in the intercalated discs (see Figure \ref{cardio}). Gap junctions are inter-cellular channels composed by hemichannels of specialized proteins, known as connexions, that control the passage of ions between neighboring cells.
 \begin{figure}[h!]
  \centering
  \includegraphics[width=13cm]{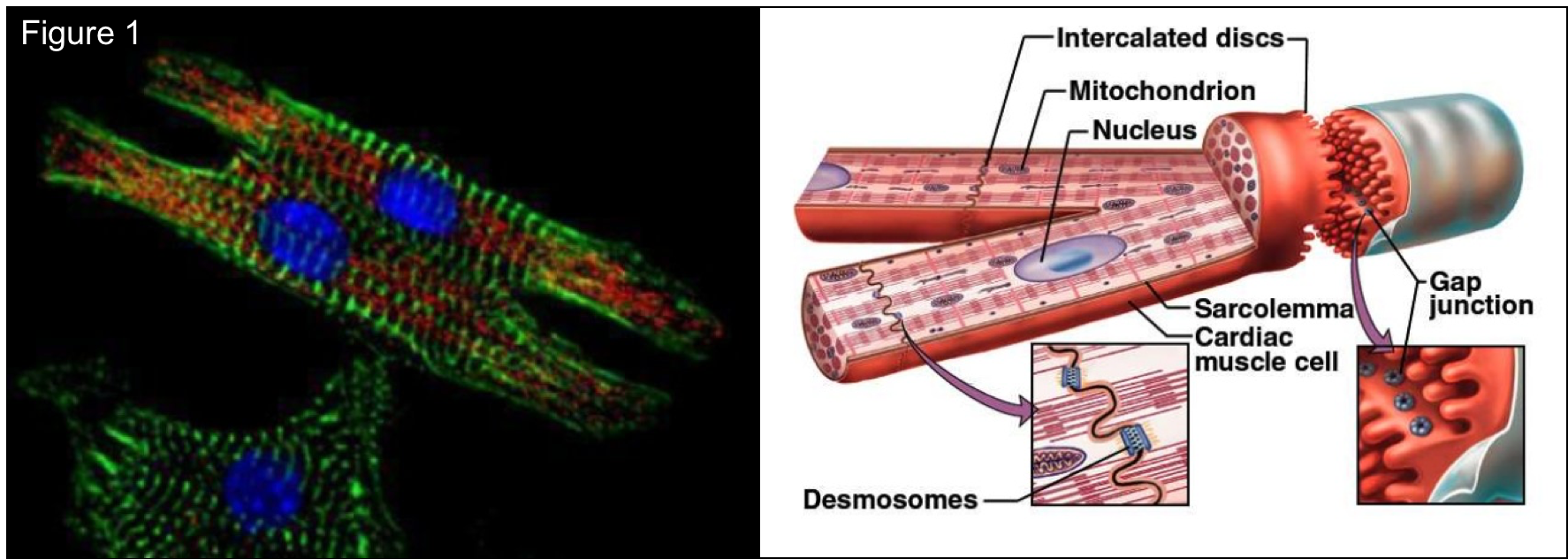}
  \caption{Representation of the cardiomyocyte structure}
  \url{http://www.cardio-research.com/cardiomyocytes}
  \label{cardio}
 \end{figure}
 
  Starting from a more accurate microscopic (cell-level) model of cardiac tissue, with the heterogeneity of the underlying cellular geometry represented in great detail, it is possible to derive the macroscopic tridomain model (tissue-level) using the homogenization method. The microscopic tridomain model consists of three quasi-static equations, two for the electrical potential in the intracellular medium and one for the extracellular medium, coupled by ordinary differential equations describing the dynamics of the ions channels  at each membrane (the sarcolemma) and at gap junctions. These equations depend on scaling parameter $\e$ whose is the ratio of the microscopic scale from the macroscopic one. The microscopic tridomain model was proposed three years ago \cite{tveito17,tveito19} in the case of just two coupled cells. Recently, we have extended in \cite{BaderTridPart1} this microscopic tridomain model to larger collections of cells. Further, we have established the well-posedness of this  problem and proved the existence and uniqueness of their solutions based on Faedo-Galerkin method.

 The macroscopic tridomain model is used as a quantitative description of the electric activity in cardiac tissue with dynamical gap junctions. The relevant unknowns are the two intracellular $u_i^k$ for $k=1,2$ and extracellular $u_e$ potentials, along with the so-called transmembrane potential $v^k := u_i^k-u_e$ for $k=1,2$ and the so-called gap potential $s := u_i^1-u_i^2$. In this model, the intra- and extracellular spaces are considered at macro-scale as two separate homogeneous domains superimposed on the cardiac domain. Conduction of electrical signals in cardiac tissue relies on the flow of ions through cell membrane and gap junctions. Each intracellular domain and extracellular one are separated by the cell membrane while the two intracellular domains are connected by gap junctions (see Figure \ref{two_scale_gap}).  The macroscopic tridomain model can be viewed as a PDE system consisting of three degenerate reaction-diffusion equations involving the unknowns $(u_i^1$, $u_i^2$, $u_e)$. These equations are supplemented by a ODE system for the dynamics of the ion channels through the cell membrane (involving the gating variable $w^k$ for $k=1,2$).

 Regarding the classical bidomain model in the literature, there are formal and rigorous mathematical derivations of the macroscopic model from a microscopic description of heart tissue. From a mathematical point of view, Krassowska et al. \cite{neukra} applied the two-scale method to formally obtain this macroscopic model (see also \cite{amar06,henri} for different approaches). Furthermore, Pennachio et al. \cite{colli05} used the tools of the $\Gamma$-convergence method to obtain a rigorous mathematical form of this homogenized macroscopic model. Amar et al. \cite{amar13} studied a hierarchy of electrical conduction problems in biological tissues via two-scale convergence. While, the authors in \cite{bendunf19} proved the existence and uniqueness of solution of the microscopic bidomain model based on Faedo-Galerkin technique. Further, they used the periodic unfolding method at two scales to show that the solution of the microscopic biodmain model converges to the solution of the macroscopic one. Recently, we have developed the meso-microscopic bidomain model by taking account three different scales and derived a new approach of its macroscopic model using two different homogenization methods. The first method \cite{BaderDev} is a formal and intuitive method based on a new three-scale asymptotic expansion method applied to the meso- and microscopic model. The second one \cite{BaderUnf} based on unfolding operators which not only derive the homogenized equation but also prove the convergence and rigorously justify the mathematical writing of the preceding asymptotic expansion method.

  \textit{The main contribution of our paper} is to provide a simple homogenization proof that can handle some relevant nonlinear membrane models (the FitzHugh-Nagumo model), relying only on unfolding operators. More precisely, we show that the solution constructed in the microscopic tridomain problem converge to the solution of the macroscopic (homogenized) tridomain model. So, we will derive the homogenized tridomain model of cardiac electro-physiology from the microscopic one using the periodic unfolding technique. The latter method not only makes it possible to derive the homogenized equation but also to prove the convergence and to rigorously justify the mathematical writing of the preceding formal method. The homogenization method that we propose allows us to investigate the effective properties of the cardiac tissue at each structural level, namely, micro-macro scales.  
  
  \textit{The paper is organized as follows:} Section \ref{geotrid} is devoted to the geometrical setting and to the introduction of the microscopic tridomain problem. In Section \ref{main_results_trid},  we state our main homogenization results. Next, some notations and properties on the domain and boundary unfolding operators are introduced in Section \ref{time-depending operators}. Finally, Section \ref{methodunf_gap} is devoted to  homogenization procedure based on unfolding operators.

\section{Tridomain modeling of the heart tissue}\label{geotrid}
The aim of this section is to describe the geometry of the cardiac tissue and to present the microscopic tridomain model of the heart.
\subsection{Geometrical setting of heart tissue}
 Let $\Omega$ be an open connected bounded subset of $\R^d,$ $d \geq 3$. The typical periodic geometrical setting is displayed in Figure \ref{two_scale_gap}.  

 \begin{figure}[h!]
  \centering
  \includegraphics[width=12cm]{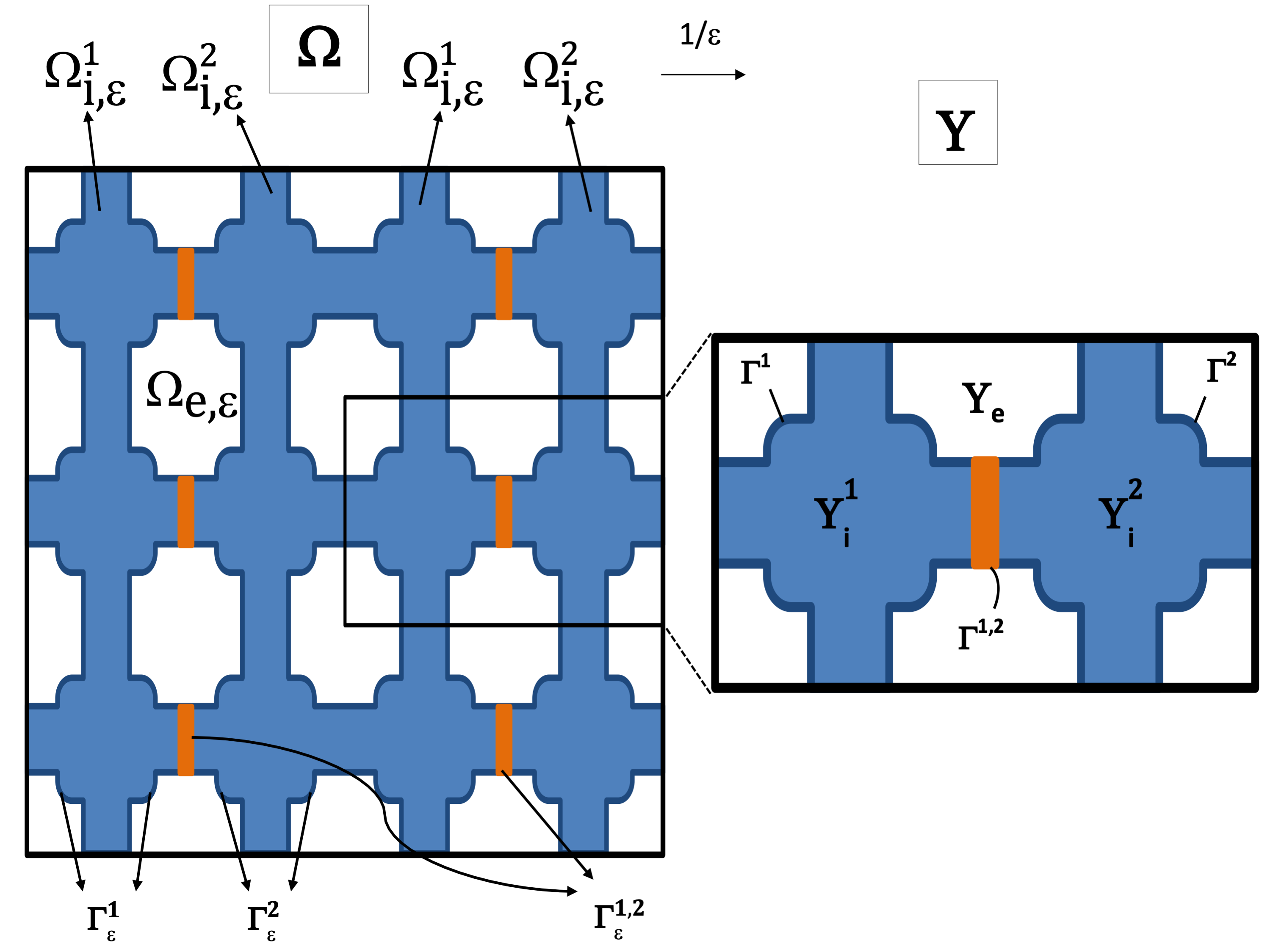}
  \caption{(Left) Periodic heterogeneous domain $\Omega.$ (Right) Reference cell $Y$ at $\e$-structural level.}
  \label{two_scale_gap}
 \end{figure} 

Let $\e\in (0,1)$ be a small positive parameter, related to the characteristic dimension of the micro-structure and which takes values in a sequence of strictly positive numbers tending to zero. Under the one-level scaling, the characteristic length $\ell^\text{mic}$ is related to a given macroscopic  length $L$ (of the cardiac fibers), such that the scaling parameter $\e$ introduced by:
 $$\e=\frac{\ell^\text{mic}}{L}.$$

From the biological point of view, the cardiac cells are connected by many gap junctions. Therefore, geometrically, $\Omega$ represents the region occupied by the cardiac tissue and  consists of two intracellular media $\Omega_{i,\e}^{k}$ for $k=1,2,$ that are connected by gap junctions $\Gamma^{1,2}_{\e}=\pt \Omega_{i,\e}^{1} \cap \pt \Omega_{i,\e}^{2}$ and extracellular medium $\Omega_{e,\e}$ (for more details see \cite{tveito17,tveito19}). 
 Each intracellular medium $\Omega_{i,\e}^{k}$ and the extracellular one  $\Omega_{e,\e}$ are separated by the surface membrane $\Gamma_{\e}^{k}$ (the sarcolemma) which is expressed by: $$\Gamma_{\e}^{k}=\pt \Omega_{i,\e}^{k} \cap \pt \Omega_{e,\e}, \text{ with } k=1,2,$$ while the remaining (exterior) boundary is denoted by $\pt_{\text{ext}} \Omega$. We can consider that the intracellular zone as a perforated domain obtained from $\Omega$ by removing the holes which correspond to the extracellular domain $\Omega_{e,\e}.$ 
 
  We can divide $\Omega$ into $N_{\e}$ small elementary cells $Y_{\e}=\overset{d}{\underset{n=1}{\prod }}]0,\e \, \ell^\text{mic}_n[,$  with $\ell^\text{mic}_1,\dots,\ell^\text{mic}_d$ are positive numbers. These small cells are all equal, thanks to a translation and scaling by $\e,$ to the same reference cell of periodicity called the reference cell  $Y=\overset{d}{\underset{n=1}{\prod }}]0,\ell^\text{mic}_n[.$ So, the $\e$-dilation of the reference cell $Y$ is defined as the following shifted set $Y_{\e,h}:$
\begin{equation}
 Y_{\e,h}:=T^h_\e+\e Y=\lbrace \e \xi : \xi \in h_\ell+Y \rbrace,
\label{trans_Y}
 \end{equation}
 where $T_\e^h$ represents the translation of $\e h$ with $h=( h_1,\dots, h_d ) \in\mathbb{Z}^d$ and  $h_\ell:=( h_1\ell^\text{mic}_1,\dots,  h_d \ell^\text{mic}_d ).$\\
 Therefore, for each macroscopic variable $x$ that belongs to $\Omega,$ we define the corresponding microscopic variable $y\approx\dfrac{x}{\e}$ that belongs to $Y$ with a  translation. Indeed, we have:
 \begin{equation*}
 x \in \Omega \Rightarrow \exists h \in\mathbb{Z}^d  \ \text{ such that }  \ x \in Y^h_\e \Rightarrow x=\e (h_\ell+y) \Rightarrow y=\dfrac{x}{\e}-h_\ell \in Y.
 \end{equation*}

 Since we will study the behavior of the functions $u(x,y)$ which are $\textbf{y}$-periodic, by periodicity we have $u\left( x,\dfrac{x}{\e}-h_\ell\right)  =u\left( x,\dfrac{x}{\e}\right) .$ By construction, we say that $y=\dfrac{x}{\e}$ belongs to $Y.$

 We are assuming that the cells are periodically organized as a regular network of interconnected cylinders at the microscale. The microscopic reference cell $Y$ is also divided into three disjoint connected parts: two intracellular parts $Y_i^{k}$ for $k=1,2,$ that are connected by an intercalated disc (gap junction) $\Gamma^{1,2}$ and extracellular part $Y_e.$ Each intracellular parts $Y_i^k$ and the extracellular one are separated by a common boundary $\Gamma^{k}$ for $k=1,2.$ So, we have:
 \begin{equation*}
 Y:=\overline{Y}_i^{1} \cup \overline{Y}_i^{2} \cup \overline{Y}_e, \quad \Gamma^{k}:= \pt Y_i^{k} \cap \pt Y_e,\quad\Gamma^{1,2}:= \pt Y_i^{1} \cap \pt Y_i^{2},
 \end{equation*}
  with $k=1,2.$ 
In a similar way, we can write the corresponding common periodic boundary as follows:
  \begin{equation}
\Gamma_{\e,h}=T^h_\e+\e \Gamma=\lbrace \e \xi : \xi \in h_\ell+\Gamma \rbrace,
\label{trans_gamma}
 \end{equation}
with $T^h_\e$ denote the same previous translation, $\Gamma_{\e,h}:=\Gamma^{k}_{\e,h},\Gamma^{1,2}_{\e,h}$ and $\Gamma:=\Gamma^{k},\Gamma^{1,2}$ for $k=1,2$.

 In summary, the intracellular and extracellular media can be described as follows:  
 \begin{equation}
 \begin{aligned}
&\Omega_{i,\e}^{k}=\Omega \cap \underset{h\in \mathbb{Z}^d}{\bigcup} Y^k_{i,\e,h}, \quad \Omega_{e,\e}=\Omega \cap \underset{h\in \mathbb{Z}^d}{\bigcup} Y_{e,\e,h},
\\& \quad \Gamma_{\e}^{k}=\Omega \cap \underset{h\in \mathbb{Z}^d}{\bigcup} \Gamma^{k}_{\e,h} \text{ and } \Gamma_{\e}^{1,2}=\Omega \cap \underset{h\in \mathbb{Z}^d}{\bigcup} \Gamma^{1,2}_{\e,h},
\end{aligned}
\label{intra_extra_domains}
\end{equation}
where $Y^k_{i,\e,h},$ $Y_{e,\e,h}$ and $\Gamma_{\e}^{k},\Gamma_{\e}^{1,2}$ are respectively defined as \eqref{trans_Y}-\eqref{trans_gamma} for $k=1,2$. 
\subsection{Microscopic tridomain model}

  The electric properties of the tissue at cellular level are described by the intracellular $u_{i,\e}^{k}$ for $k=1,2$ and extracellular
$u_{e,\e}$, potentials respectively with the associated conductivities $\mathrm{M}_{i}^{\e}$ and $\mathrm{M}_{e}^{\e}$. In \cite{BaderTridPart1}, we presented and studied in details the non-dimensional tridomain model with respect the scaling parameter $\e$, as well as the models chosen for the membrane and gap junctions dynamics.  More precisely, we consider the following microscopic tridomain model:
 
\begin{subequations}
\begin{align}
-\nabla\cdot\left( \mathrm{M}_{i}^{\e}\nabla u_{i,\e}^{k}\right)  &=0 &&\text{ in } \Omega_{i,\e,T}^{k}:=(0,T)\times\Omega_{i,\e}^{k}, 
\label{trid_intra}
\\ -\nabla\cdot\left( \mathrm{M}_{e}^{\e}\nabla u_{e,\e}\right)  &=0 &&\text{ in } \Omega_{e,\e,T}:=(0,T)\times\Omega_{e,\e}, 
\label{trid_extra}
\\u_{i,\e}^{k}-u_{e,\e}&=v_{\e}^{k} &&\ \text{on} \ \Gamma_{\e,T}^{k}:=(0,T)\times\Gamma_{\e}^{k},
\label{trid_v}
\\-\mathrm{M}_{i}^{\e}\nabla u_{i,\e}^{k} \cdot n_i^{k}=\mathrm{M}_{e}^{\e}\nabla u_{e,\e} \cdot n_{e} & =\I_{m}^{k} &&\ \text{on} \ \Gamma_{\e,T}^{k},
\label{trid_cont_sarco}
\\ \e\left( \pt_{t} v_{\e}^{k}+\I_{ion}\left(v_{\e}^{k},w_{\e}^{k}\right) -\I_{app,\e}^{k}\right) &=\I_{m}^{k} &&\ \text{on} \ \Gamma_{\e,T}^{k},
\label{trid_bord_sarco}
\\ \pt_{t} w_{\e}^{k}-H\left(v_{\e}^{k},w_{\e}^{k}\right) &=0 && \text{ on }  \Gamma_{\e,T}^{k},
\label{trid_dyn_sarco}
\\u_{i,\e}^{1}-u_{i,\e}^{2}&=s_{\e} &&\ \text{on} \ \Gamma_{\e,T}^{1,2}:=(0,T)\times\Gamma_{\e}^{1,2},
\label{trid_s}
\\-\mathrm{M}_{i}^{\e}\nabla u_{i,\e}^{1} \cdot n_i^{1}=\mathrm{M}_{i}^{\e}\nabla u_{i,\e}^{2} \cdot n_i^{2}& =\I_{1,2} &&\ \text{on} \ \Gamma_{\e,T}^{1,2},
\label{trid_cont_gap}
\\ \frac{\e}{2}\left( \pt_{t} s_{\e}+\I_{gap}\left(s_{\e}\right) \right) &=\I_{1,2} &&\ \text{on} \ \Gamma_{\e,T}^{1,2},
\label{trid_bord_gap}
\end{align}
\label{pbscale_gap}
\end{subequations}
with $k=1,2$ and each equation corresponds to the following sense:
\eqref{trid_intra} Intra quasi-stationary conduction, \eqref{trid_extra} Extra quasi-stationary conduction, \eqref{trid_v} Transmembrane potential, \eqref{trid_cont_sarco} Continuity equation at cell membrane, \eqref{trid_bord_sarco} Reaction condition at the corresponding cell membrane, \eqref{trid_dyn_sarco} Dynamic coupling, \eqref{trid_s} Gap junction potential, \eqref{trid_cont_gap} Continuity equation at gap junction, \eqref{trid_bord_sarco} Reaction condition at gap junction. 

  Observe that the tridomain equations \eqref{trid_intra}-\eqref{trid_extra} are invariant with respect to the scaling parameter $\e$. As usual in homogenization theory, the electrical potentials are assumed to have the following form
$$u_{i,\e}^{k}(t,x):=u_{i}^{k}\left( t, x,\frac{x}{\e}\right), \ \ u_{e,\e}(t,x):= u_{e}\left( t, x,\frac{x}{\e}\right), \text{ for } k=1,2,$$
where each function depends on time $t \in (0,T)$, slow (macroscopic) variable $x$ and the fast (microscopic) variable $y=x/\e$.
Similarly, the transmembrane potential $v_{\e}^{k},$  the gap junction potential $s_{\e}$ and the corresponding gating variable $w_{\e}^{k}$ for $k=1,2$ have the same previous form. Furthermore, the conductivity tensors are considered symmetric and dependent both on the slow and fast variables, i.e. for $j=i,e,$ we have 
\begin{equation}
\mathrm{M}_{j}^{\e}(x):=\mathrm{M}_{j}\left( x,\frac{x}{\e}\right),
\label{M_ie_gap}
\end{equation}
satisfying the elliptic and periodicity conditions: there exist constants $\alpha, \beta \in \R,$ such that $0<\alpha<\beta$ and for all $\lambda\in \R^d:$
\begin{subequations}
\begin{align}
&\mathrm{M}_j\lambda\cdot\lambda \geq \alpha\abs{\lambda}^2, 
\\& \abs{\mathrm{M}_j\lambda}\leq \beta \abs{\lambda},
\\& \mathrm{M}_j \ \mathbf{y}\text{-periodic},\text{ for }  j=i,e.
\end{align}
\label{A_Mie_gap}
\end{subequations}
We complete system \eqref{pbscale_gap} with no-flux boundary conditions on $\pt_{\text{ext}} \Omega$: 
\begin{equation*}
\left( \mathrm{M}_{i}^{\e}\nabla u_{i,\e}^{k}\right) \cdot \mathbf{n}=\left( \mathrm{M}_{e}^{\e}\nabla u_{e,\e}\right) \cdot\mathbf{n}=0 \ \text{ on } \  (0,T)\times \pt_{\text{ext}} \Omega,
\end{equation*}
where $k=1,2$ and $\mathbf{n}$ is the outward reference normal to the exterior boundary of $\Omega.$ We impose initial conditions on transmembrane potential $v_{\e}^{k},$ gap junction potential $s_{\e}$ and gating variable $w_{\e}^{k}$ as follows: 
\begin{equation}
\begin{aligned}
& v_{\e}^{k}(0,x)=v_{0,\e}^{k}(x), \ w_{\e}^{k}(0,x)=w_{0,\e}^{k}(x) & \text{ a.e. on } \Gamma_{\e,T}^{k}, 
\\ & \text{and } s_{\e}(0,x)=s_{0,\e}(x) & \text{ a.e. on } \Gamma_{\e,T}^{1,2},
\end{aligned}
\label{cond_ini_vws_gap}
\end{equation}
with $k=1,2.$ \\ \\
Next, we introduce some assumptions on the ionic functions, the source term and the initial data.\\ 
\textbf{Assumptions on the ionic functions.} The ionic current $\I_{ion}(v^{k},w^{k})$ at each cell membrane $\Gamma^{k}$ can be decomposed into $\mathrm{I}_{a,ion}\left( v^{k} \right)$ and $\mathrm{I}_{b,ion}\left( w^{k}\right) ,$ where $\I_{ion}\left(v^{k},w^{k}\right)=\mathrm{I}_{a,ion}\left(v^{k}\right) +\mathrm{I}_{b,ion}\left(w^{k}\right) $ with $k=1,2.$ Furthermore, the nonlinear function $\mathrm{I}_{a,ion}: \R \rightarrow \R$ is considered as a $C^1$ function and the functions $\mathrm{I}_{b,ion}: \R \rightarrow \R$  and $H : \R^2 \rightarrow \R$ are considered as linear functions. Also, we assume that there exists $r\in (2,+\infty)$ and constants $\alpha_1,\alpha_2,\alpha_3, \alpha_4, \alpha_5, C>0$ and $\beta_1>0, \beta_2\geq 0$ such that:

\begin{subequations}
\begin{align}
&\dfrac{1}{\alpha_1} \abs{v}^{r-1}\leq \abs{\mathrm{I}_{a,ion}\left(v\right)}\leq \alpha_1\left(\abs{v}^{r-1}+1\right), \,\abs{\mathrm{I}_{b,ion}\left(w\right)}\leq \alpha_2(\abs{w}+1), 
\label{A_I_ab} 
\\& \abs{H(v,w)}\leq \alpha_3(\abs{v}+\abs{w}+1),\text{ and }
    \mathrm{I}_{b,ion}\left(w\right)v-\alpha_4 H(v,w)w\geq \alpha_5 \abs{w}^2,
\label{A_H_Ib_a}
\\& \tilde{\mathrm{I}}_{a,ion} : v\mapsto \mathrm{I}_{a,ion}(v)+\beta_1 v+\beta_2 \text{ is strictly increasing with } \lim \limits_{v\rightarrow 0} \tilde{\mathrm{I}}_{a,ion}(v)/v=0,
\label{A_tildeI_a_1} 
\\& \forall v,v' \in \R,\,\,\left(\tilde{\mathrm{I}}_{a,ion}(v)-\tilde{\mathrm{I}}_{a,ion}(v') \right)(v-v')\geq \dfrac{1}{C} \left(1+\abs{v}+\abs{v'} \right)^{r-2} \abs{v-v'}^{2},
\label{A_tildeI_a_2} 
\end{align}
\label{A_H_I_gap}
\end{subequations}
with $(v,w):=\left(v^{k},w^{k}\right)$ for $k=1,2.$

Now, we represent the gap junction $\Gamma_{\e}^{1,2}$ between intra-neighboring cells by a passive membrane:
\begin{equation}
\I_{gap}(s)=G_{gap}\times s,
\label{gap_model}
\end{equation}
where $G_{gap}=\frac{1}{R_{gap}}$ is the conductance of the gap junctions. A discussion of the modeling of the gap junctions is given in \cite{hogues}.
 \\ \\
\textbf{Assumptions on the source term.} There exists a constant $C$ independent of $\e$ such that the source term $\I_{app,\e}^{k}$ satisfies the following estimation for $k=1,2$:
\begin{equation}
\norm{\e^{1/2}\I_{app,\e}^{k}}_{L^{2}(\Gamma_{\e,T}^{k})}\leq C.
\label{A_iapp_gap}
\end{equation}  
\textbf{Assumptions on the initial data.} The initial condition $v_{0,\e}^{k},$ $s_{0,\e}$ and $w_{0,\e}^{k}$ satisfy the following estimation:
\begin{equation}
\sum\limits_{k=1,2}\norm{\e^{1/r}v_{0,\e}^{k}}_{L^{r}(\Gamma_{\e}^{k})}+\norm{\e^{1/2}s_{0,\e}}_{L^{2}(\Gamma_{\e}^{1,2})}+\sum\limits_{k=1,2}\norm{\e^{1/2}w_{0,\e}^{k}}_{L^{2}(\Gamma_{\e}^{k})}\leq C,
\label{A_vw0_gap}
\end{equation}
for some constant $C$ independent of $\e.$ Moreover, $v_{0,\e}^{k},$ $s_{0,\e}$ and $w_{0,\e}^{k}$ are assumed to be traces of uniformly bounded sequences in $C^{1}(\overline{\Omega})$ with $k=1,2.$

  Finally,  we observe that the equations in \eqref{pbscale_gap} are invariant under the change of $u_{i,\e}^{k},$ $k=1,2$ and $u_{e,\e}$ into $u_{i,\e}^{k}+c,$ $u_{e,\e}+c,$ for any $c\in\R.$ Therefore, we may impose the following normalization condition:
 \begin{equation}
 \int_{\Omega_{e,\e}} u_{e,\e} \ dx=0, \text{ for a.e. } t\in(0,T).
 \label{normalization}
 \end{equation}

\section{Main results}\label{main_results_trid}
 In this part, we highlight the main results obtained in our paper. Based on the a priori estimates and unfolding homogenization method, we can pass to the limit in the microscopic equations and derive the following homogenized problem: 
\begin{thm}[Macroscopic Tridomain Model]
Assume that conditions \eqref{A_Mie_gap}-\eqref{normalization} hold. Then, a sequence of solutions $\Bigl(u_{i,\e}^{1},u_{i,\e}^{2},u_{e,\e},w_{\e}^{1}, w_{\e}^{2}\Bigl)_{\e}$ of the microscopic tridomain model \eqref{pbscale_gap} converges as $\e \to 0$ to a weak solution $\Bigl(u_{i}^{1},u_{i}^{2},u_{e},w^{1}, w^{2}\Bigl)$ satisfying the following conditions:

\begin{enumerate}[label=(\Alph*)]
\item (Algebraic relation).
\begin{equation*}
\begin{aligned}
v^{k}&=u_{i}^{k}-u_{e}  &\ \text{ for } k=1,2, \text{a.e. in} \ \Omega_{T},
\\ s &=u_{i}^{1}-u_{i}^{2} &\ \text{a.e. in} \ \Omega_{T}.
\end{aligned}
\end{equation*}
\item (Regularity).
\begin{equation*}
\begin{aligned}
 & u_i^{k},u_e\in L^2(0,T;H^1(\Omega)),
 \\& \int_{\Omega} u_{e}(t,x) \ dx=0, \text{ for a.e. } t\in(0,T),
\\ & v^{k} \in L^2(0,T;H^1(\Omega))\cap L^r(\Omega_T), \ r\in (2,+\infty),
\\ & s \in L^2(0,T;H^1(\Omega)), 
\\ & w^{k}\in C(0,T;L^2(\Omega)),  
\\ & \partial_t v^{k} \in L^2(0,T;(H^1(\Omega))')+L^{r/(r-1)}(\Omega_T),
\\ & \pt_t s, \pt_t w^{k} \in L^2(\Omega_T), \ k=1,2.
\end{aligned}
\end{equation*}
\item (Initial conditions).
\begin{equation*}
\begin{aligned}
& v^{k}(0,x)=v_{0}^{k}(x), \ w^{k}(0,x)=w_{0}^{k}(x), \ k=1,2 & \text{ a.e. in } \Omega, 
\\ & \text{and } s(0,x)=s_{0}(x) & \text{ a.e. in } \Omega.
\end{aligned}
\end{equation*}
\item (Boundary conditions).
\begin{equation*}
\left(\widetilde{\mathbf{M}}_{e}\nabla u_{e}\right)\cdot\mathbf{n}=\left( \widetilde{\mathbf{M}}_{i}\nabla u_{i}^{k}\right) \cdot\mathbf{n}=0 \ \text{ on } \  \Sigma_T:=(0,T)\times\pt_{\text{ext}} \Omega,
\end{equation*}
\item (Differential equations).
\begin{equation}
\begin{aligned}
\sum \limits_{k=1,2} \mu_{k}\pt_{t} v^{k}+\nabla \cdot\left(\widetilde{\mathbf{M}}_{e}\nabla u_{e}\right) +\sum \limits_{k=1,2}\mu_{k}\I_{ion}(v^{k},w^{k}) &= \sum \limits_{k=1,2}\mu_{k}\I_{app}^{k} &\text{ in } \Omega_{T},
\\ \mu_{1}\pt_{t} v^{1}+ \mu_{g}\pt_{t} s-\nabla \cdot\left( \widetilde{\mathbf{M}}_{i}\nabla u_{i}^{1}\right)+\mu_{1}\I_{ion}(v^{1},w^{1})+\mu_{g}\I_{gap}(s) &= \mu_{1} \I_{app}^{1} &\text{ in } \Omega_{T},
\\ \mu_{2}\pt_{t} v^{2}- \mu_{g}\pt_{t} s-\nabla \cdot\left( \widetilde{\mathbf{M}}_{i}\nabla u_{i}^{2}\right)+\mu_{2}\I_{ion}(v^{2},w^{2})-\mu_{g}\I_{gap}(s) &= \mu_{2} \I_{app}^{2} &\text{ in } \Omega_{T},
\\ \pt_{t} w^{k}-H(v^{k},w^{k}) &=0 & \text{ in }  \Omega_{T},
\label{pb_macro_gap}
\end{aligned}
\end{equation}
\end{enumerate}
where $\mu_{k}=\abs{\Gamma^{k}}/\abs{Y},$ $k=1,2,$ $\Big($resp. $\mu_{g}=\abs{\Gamma^{1,2}}/2\abs{Y}\Big)$ is the ratio between the surface membrane (resp. the gap junction) and the volume of the reference cell. Furthermore, $\mathbf{n}$ represents the outward reference normal to the boundary of $\Omega.$ Herein, the homogenized conductivity matrices $\widetilde{\mathbf{M}}_j=\left( \widetilde{\mathbf{m}}^{pq}_j\right)_{1\leq p,q \leq d}$ for $j=i,e$ are respectively defined by:
\begin{subequations}
\begin{align}
& \widetilde{\mathbf{m}}^{pq}_i:=\dfrac{1}{\abs{Y}}\overset{d}{\underset{\ell=1}{\sum}}\displaystyle\int_{Y_{i}^{k}}\left( \mathrm{m}_i^{pq}+\mathrm{m}^{p\ell}_{i}\dfrac{\pt \chi_i^q}{\pt y_\ell}\right) \ dy,
\label{tilde_m_i_gap}
\end{align}
\begin{align}
&\widetilde{\mathbf{m}}^{pq}_e:=\dfrac{1}{\abs{Y}}\overset{d}{\underset{\ell=1}{\sum}}\displaystyle\int_{Y_e}\left( \mathrm{m}_e^{pq}+\mathrm{m}^{p\ell}_{e}\dfrac{\pt \chi_e^q}{\pt y_\ell}\right) \ dy,
\label{tilde_m_e_gap}
\end{align}
\end{subequations}
where the components $\chi_{j}^{q}$ of $\chi_j$ for $j=i,e$ are respectively the corrector functions, solutions of the $\e$-cell problems:
\begin{subequations}
\begin{equation}
\begin{cases}
-\nabla_{y}\cdot\left( \mathrm{M}_{e}\nabla_{y} \chi_{e}^q\right) =\nabla_{y}\cdot\left( \mathrm{M}_{e} e_{q}\right) \ \text{in} \ Y_{e},
\\ \chi_e^q \ y\text{-periodic}, 
\\ \mathrm{M}_{e} \nabla_y \chi_e^q \cdot n_e= - (\mathrm{M}_{e}e_q )\cdot n_e \text{ on } \Gamma^{k}, \ k=1,2
\end{cases}
\end{equation}
\begin{equation}
\begin{cases}
-\nabla_{y}\cdot\left( \mathrm{M}_{i}\nabla_{y} \chi_{i}^q\right) =\nabla_{y}\cdot\left( \mathrm{M}_{i} e_{q}\right) \ \text{in} \ Y_{i}^{k},
\\ \chi_{i}^{q} \ y\text{-periodic},
 \\ \mathrm{M}_{i} \nabla_y \chi_i^q \cdot n_{i}^{k}= - (\mathrm{M}_{i}e_q )\cdot n_{i}^{k} \text{ on } \Gamma^{k}, \ k=1,2 
\\ \mathrm{M}_{i} \nabla_y \chi_i^q \cdot n_{i}^{k}= - (\mathrm{M}_{i}e_q )\cdot n_{i}^{k} \text{ on } \Gamma^{1,2},  
 \end{cases}
 \end{equation}
 \end{subequations}
for $e_q$, $q=1,\dots,d,$ the standard canonical basis in $\R^d.$
\label{thm_macro_gap}
\end{thm}

The proof of Theorem \ref{thm_macro_gap} is proved rigorously in Section \ref{unf_gap} using unfolding homogenization method. The uniqueness of the solutions to the macroscopic model can be proved similar as that of the microscopic model with minor changes (see \cite{BaderTridPart1}). This implies that all the convergence results remain valid for the whole sequence. Furthermore, it is easy to verify that the macroscopic conductivity tensors of the intracellular and extracellular spaces are symmetric and positive definite (see Remark \ref{Mt_pos_ell}).

\begin{rem} The authors in \cite{bendunf19} treated the microscopic bidomain problem where the gap junction is ignored. They considered that there are only two intra- and extracellular media separated by a single membrane (sarcolemma). Comparing to \cite{bendunf19}, the microscopic tridomain model in our work consists of three elliptic equations coupled through three boundary conditions, two on each cell membrane and one on the gap junction which separates between two intracellular media. The macroscopic tridomain model is more general and complex than the classical monodomain and bidomain models. Using periodic unfolding homogenization method, we derive a new approach of the homogenized model \eqref{pb_macro_gap} from the microscopic tridomain problem \eqref{pbscale_gap}. 
\end{rem}

\begin{rem} Regarding the classical bidomain model \cite{colli12,bendunf19}, we can derive this model from our tridomain problem if we take $u_{i}^{1}=u_{i}^{2}$. 

\end{rem}

\section{Time-depending unfolding operators}\label{time-depending operators}
 \subsection{Unfolding operator and some basic properties}\label{unfop_gap} Under the notation \eqref{intra_extra_domains}, we begin with introducing the unfolding operator and describe some of its properties. For more properties and proofs, we refer to \cite{doinaunf18,doinaunf12}. First, we present the unfolding operators defined for perforated domains on the domain $(0,T)\times\Omega.$ Then we define boundary unfolding operators one on the membrane $(0,T)\times\Gamma^{k},$ $k=1,2$ and the other on the gap junction $(0,T)\times\Gamma^{1,2}$.
		
 In order to define an unfolding operator, we first introduce the following sets in $\R^d$ (see Figure \ref{fig_unf_gap})
 \begin{itemize}
 \item $\Xi_\e=\lbrace h\in \mathbb{Z}^d, \ \e(h_\ell+Y)\subset\Omega\rbrace,$
 \\
 \item $\widehat{\Omega}_{\e}=$ interior $\lbrace\underset{h \in \Xi_\e }{\bigcup} \e \left(h_\ell +\overline{Y}\right) \rbrace,$
 \\
 \item $\widehat{\Omega}_{e,\e}= $ interior $\lbrace\underset{h \in \Xi_\e }{\bigcup} \e \left(h_\ell +\overline{Y_e}\right)\rbrace,$
 \\
 \item $\widehat{\Omega}_{i,\e}^{k}=$ interior $\lbrace\underset{h \in \Xi_\e }{\bigcup} \e \left( h_\ell +\overline{Y_i^{k}}\right)\rbrace, \ k=1,2,$
 \\
  \item $\widehat{\Gamma}_{\e}^{k}=\lbrace y\in \Gamma^{k} : y\in \widehat{\Omega}_{\e}\rbrace, \ k=1,2,$
  \\
  \item $\widehat{\Gamma}_{\e}^{1,2}=\lbrace y\in \Gamma^{1,2} : y\in \widehat{\Omega}_{\e}\rbrace,$
 \\
 \item $\Lambda^\e=\Omega\setminus\widehat{\Omega}^{\e},$
 \\
 \item $\widehat{\Omega}_{\e,T}=(0,T)\times \widehat{\Omega}^{\e},$
 \\
 \item $\widehat{\Omega}_{i,\e,T}^{k}=(0,T)\times \widehat{\Omega}_{i,\e}^{k}, \ k=1,2, \qquad \widehat{\Omega}_{e,\e,T}=(0,T)\times \widehat{\Omega}_{e,\e}, $
 \\
 \item $\Lambda_{T}^\e=(0,T)\times \Lambda^\e,$
 \end{itemize}
 where $h_\ell:=( h_1\ell^\text{mic}_1,\dots,  h_d \ell^\text{mic}_d ).$
 \begin{figure}[h!]
  \centering
  \includegraphics[width=11cm]{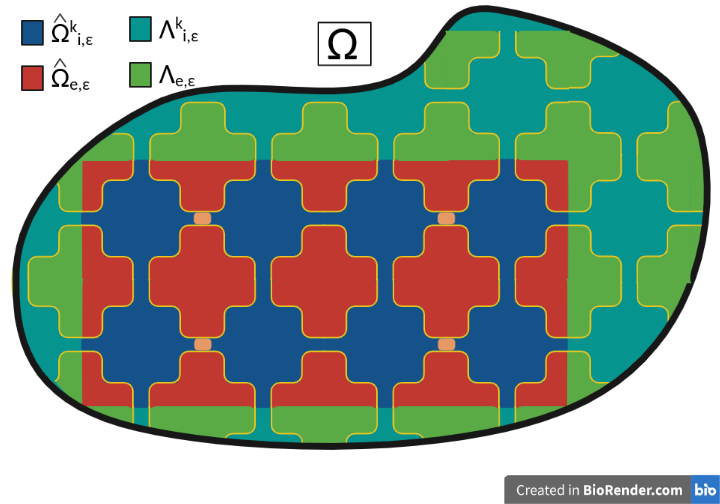}
  \caption{The sets $\widehat{\Omega}_{i,\e}^{k}$ for $k=1,2$ (in blue), $\widehat{\Omega}_{e}^{\e}$ (in red), $\Lambda_{i,\e}^{k}$ (in dark cyan)  and $\Lambda_{e,\e}$ (in green).}
  \label{fig_unf_gap}
 \end{figure} 
 For all $w\in \R^d,$ let $[w]_{Y}$ be the unique integer combination of the periods such that $w-[w]_{Y} \in Y.$ We may write $w=[w]_{Y}+\lbrace w\rbrace_Y$ for all $w\in \R^d,$ so that for all $\e>0,$ we get the unique decomposition:
 $$x=\e\left(\left[ \dfrac{x}{\e} \right]_Y + \left\lbrace \dfrac{x}{\e} \right\rbrace_Y  \right), \ \text{for all } x\in \R^d.  $$ 	

 Based on this decomposition, we define the unfolding operator in intra- and extracellular domains.
 
\begin{defi}[Domain and boundary unfolding operator \cite{doinaunf18,doinaunf12}] $ $ 
\begin{itemize}
\item[$1.$]For any function $\phi$ Lebesgue-measurable on the intracellular medium $\Omega_{i,\e,T}^{k}:=(0,T)\times\Omega_{i,\e}^{k}$ for $k=1,2$, the unfolding operator $\T_{\e}^{i,k}$ is defined as follows:
\begin{equation}
\T_{\e}^{i,k}(\phi)(t,x,y)=
\begin{cases}
\phi\left(t, \e\left[ \dfrac{x}{\e} \right]_Y +\e y \right) &\text{ a.e. for } (t, x, y) \in \widehat{\Omega}^{\e}_{T}\times Y_{i}^{k}, \\ 0  &\text{ a.e. for } (t, x, y) \in \Lambda^{\e}_{T}\times Y_{i}^{k},
\end{cases}
\label{op_ie_gap}
\end{equation}
where $\left[ \cdot\right] $ denotes the Gau$\beta$-bracket.
Similarly, we define the unfolding operator $\T_{\e}^{e}$ on the domain $\Omega_{e,T}^{\e}:=(0,T)\times\Omega_{e}^{\e}.$
We readily have that:
$$\forall x \in \R^d, \ \T_{\e}^{i,k}(\phi)\left(t, x, \left\lbrace \dfrac{x}{\e} \right\rbrace_Y  \right)=\phi(t,x), \text{ with } k=1,2. $$
\item[$2.$] For any function $\varphi$ Lebesgue-measurable on the membrane $\Gamma_{\e}^{k}:=(0,T)\times \Gamma_{\e}^{k}$ for $k=1,2,$ the boundary unfolding operator $\T_{\e}^{b,k}$ is defined as follows:
\begin{equation}
\T_{\e}^{b,k}(\varphi)(t,x,y)=
\begin{cases}
\varphi\left(t, \e\left[ \dfrac{x}{\e} \right]_Y +\e y \right) &\text{ a.e. for } (t, x, y) \in \widehat{\Omega}^{\e}_{T}\times \Gamma^{k}, \\ 0  &\text{ a.e. for } (t, x, y) \in \Lambda^{\e}_{T}\times \Gamma^{k}.
\end{cases}
\label{op_bk_e_gap}
\end{equation}
Similarly, we define the boundary unfolding operator $\T_{\e}^{b,1,2}$ on the gap junction $\Gamma_{\e,T}^{1,2}:=(0,T)\times \Gamma_{\e}^{1,2}.$
\end{itemize}
\end{defi}

\subsubsection{\textbf{Properties of the unfolding operator}}
In the following proposition, we state some basic properties of the unfolding operator which will be used frequently in the next sections.
\begin{prop}[Some properties of the unfolding operator \cite{doinaunf18,doinaunf12}] $ $
\begin{enumerate}
\item \label{P_uo1_gap} The operator $\T_{\e}^{i,k}: L^{p}\left(\Omega_{i,\e,T}^{k}\right) \longrightarrow L^{p}(\Omega_T\times Y_{i}^{k})$ and $\T_{\e}^{b,k} : L^p(\Gamma_{\e,T}^{k}) \longrightarrow L^p(\Omega_T\times \Gamma^{k})$ are linear and continuous for $p\in [1,+\infty)$ and $k=1,2.$ Similarly, we have the same properties for the unfolding operator $\T_{\e}^{e}$ and for the boundary unfolding operator $\T_{\e}^{b,1,2}.$
\\
\item \label{P_uo2_gap} For $u,u' \in L^{p}\left(\Omega_{i,\e,T}^{k}\right)$ and $v,w \in L^{p}\left(\Gamma_{\e,T}^{k}\right),$ it holds that  $\T_{\e}^{i,k}(uu')=\T_{\e}^{i,k}(u)\T_{\e}^{i,k}(u')$ and $\T_{\e}^{b,k}(vw)=\T_{\e}^{b,k}(v)\T_{\e}^{b,k}(w),$ with $p \in (1,+\infty)$ and $k=1,2.$
\\
\item \label{P_uo3_gap} For $u\in L^{p}\left(\Omega_{i,\e,T}^{k}\right), p \in [1,+\infty),$ we have $$\norm{\T_{\e}^{i,k}(u)}_{L^{p}(\Omega_T\times Y_{i}^{k})}=\abs{Y}^{1/p}\norm{u \mathds{1}_{\widehat{\Omega}_{i,\e,T}^{k}}}_{L^{p}\left( \Omega_{i,\e,T}^{k}\right)}\leq \abs{Y}^{1/p}\norm{u}_{L^{p}\left( \Omega_{i,\e,T}^{k}\right)}.$$

\item \label{P_uo4_gap} For $v \in L^{p}\left(\Gamma_{\e,T}^{k}\right),$ with $p \in [1,+\infty)$ and $k=1,2.$ Then we have $$\norm{\T_{\e}^{b,k}(v)}_{L^{p}(\Omega_T\times \Gamma^{k})}=\e^{1/p}\abs{Y}^{1/p}\norm{v}_{L^{p}(\widehat{\Gamma}_{\e,T}^{k})}\leq \e^{1/p}\abs{Y}^{1/p}\norm{v}_{L^{p}(\Gamma_{\e,T}^{k})}.$$

\item \label{P_uo5_gap} Let $\phi_\e \in L^p\left(0,T; W^{1,p}\left(\Omega\right)\right),$ with $p\in [1,+\infty)$ and $k=1,2.$ If $\phi_\e \rightarrow \phi$ strongly in $L^{p}(0,T;W^{1,p}(\Omega))$ as $\e\rightarrow 0,$   then  
\begin{align*}
&\T_{\e}^{i,k}(\phi_\e) \rightarrow \phi \text{ strongly in } L^{p}(\Omega_T\times Y_{i}^{k}),
\\& \T_{\e}^{b,k}(\phi_\e) \rightarrow \phi \vert_{\Gamma^{k}} \text{ strongly in } L^{p}(\Omega_T\times \Gamma^{k}) \text{ as } \e\rightarrow 0.
\end{align*}
\item \label{P_uo6_gap} For $u\in L^{p}\left(0,T; W\left(\Omega_{i,\e}^{k}\right)\right), p \in [1,+\infty),$ it holds that $\nabla_y\T_{\e}^{i,k}(u)=\e\T_{\e}^{i,k}(\nabla_x u)$ with $k=1,2.$
\end{enumerate}
\label{prop_uo_gap}
\end{prop}


\begin{rem} The unfolding operators $\T_{\e}^{b,k}$ and $\T_{\e}^{i,k}$ for $k=1,2$ are related in the following sense:
\begin{equation*}
\T_{\e}^{b,k}(u\vert_{\Gamma_{\e}^{k}})=\T_{\e}^{i,k}(u)\vert_{\Gamma^{k}}, \quad u \in L^p\left(0,T; W^{1,p}(\Omega_{i,\e}^{k})\right), \quad k=1,2,
\end{equation*}
for $p\in (1,+\infty)$ and a.e. $t\in (0,T)$. In particular, by the standard trace theorem in $Y_{i}^{k},$ there is a constant $C$ independent of $\e$ and $t$ such that
\begin{equation*}
 \norm{\T_{\e}^{b,k}(u)}_{L^{p}\left(\Omega_T \times \Gamma^{k}\right)}^p\leq C\left( \norm{\T_{\e}^{i,k}(u)}_{L^{p}\left(\Omega_T \times Y_{i}^{k}\right)}^p+\norm{\nabla_y \T_{\e}^{i,k}(u)}_{L^{p}\left(\Omega_T \times Y_{i}^{k}\right)}^p\right).
\end{equation*}
From the properties of $\T_{\e}^{i,k}(\cdot)$ in Proposition \ref{prop_uo_gap}, it follows that
\begin{equation*}
 \norm{\T_{\e}^{b,k}(u)}_{L^{p}\left(\Omega_T \times \Gamma^{k}\right)}^p\leq C\left( \norm{u}_{L^{p}\left(\Omega_{i,\e,T}^{k}\right)}^p+\e^p\norm{\nabla u}_{L^{p}\left(\Omega_{i,\e,T}^{k}\right)}^p\right).
\end{equation*}
Similarly, the trace theorem in $Y_e$ holds for $u \in L^p\left(0,T; W^{1,p}(\Omega_{e,\e})\right)$  (which can be found as Remark 4.2 in \cite{doinaunf12}).
\label{trace_ineq_gap}
\end{rem}
In the sequel, we will define $W_{\#}^{1,p}$ the periodic Sobolev space as follows:
\begin{defi} Let $\mathcal{O}$ be a reference cell and $p\in[1,+\infty)$. Then, we define
\begin{equation}
W_{\#}^{1,p}(\mathcal{O})=\lbrace u \in W^{1,p}(\mathcal{O}) \ \text{such that} \ u \ \text{is periodic with }  \M_{\mathcal{O}}(u)=0\rbrace,
\label{W_gap}
\end{equation}
where $\M_{\mathcal{O}}(u)=\dfrac{1}{\abs{\mathcal{O}}}\displaystyle\int_{\mathcal{O}} u \ dy.$
Its duality bracket is defined by $$F(u)=(F,u)_{(W_{\#}^{1,p}(\mathcal{O}))',W_{\#}^{1,p}(\mathcal{O})}=(F,u)_{(W^{1,p}(\mathcal{O}))',W^{1,p}(\mathcal{O})},\ \forall u \in W_{\#}^{1,p}(\mathcal{O}).$$
Furthermore, by the Poincaré-Wirtinger's inequality, the Banach space $W_{\#}^{1,p}$ has the following norm:
$$\norm {u}_{W_{\#}^{1,p}(\mathcal{O})}=\norm {\nabla u}_{L^p(\mathcal{O})}, \forall u \in W_{\#}^{1,p}(\mathcal{O}).$$
{\bfseries \textsc{Notation}:} We denote $W_{\#}^{1,2}(\mathcal{O})$ by $H_{\#}^{1}(\mathcal{O})$  for $p=2.$
\end{defi}

\subsection{Microscopic tridomain model}\label{unf_gap}  

 We start by stating the weak formulation of the microscopic tridomain model as given in the following definition.
\begin{defi}[Weak formulation of microscopic system]  A weak solution to problem \eqref{pbscale_gap}-\eqref{cond_ini_vws_gap} is a collection $(u_{i,\e}^{1}, u_{i,\e}^{2}, u_{e,\e}, w_{\e}^{1}, w_{\e}^{2})$ of functions satisfying the following conditions:


\begin{enumerate}[label=(\Alph*)]
\item (Algebraic relation).
\begin{equation*}
\begin{aligned}
v_{\e}^{k}&=(u_{i,\e}^{k}-u_{e,\e})\vert_{\Gamma_{\e,T}^{k}} &\ \text{a.e. on} \ \Gamma_{\e,T}^{k}, \text{ for } k=1,2,
\\ s_{\e}&=(u_{i,\e}^{1}-u_{i,\e}^{2})\vert_{\Gamma_{\e,T}^{1,2}} &\ \text{a.e. on} \ \Gamma_{\e,T}^{1,2}.
\end{aligned}
\end{equation*}
\item (Regularity).
\begin{equation*}
\begin{aligned}
 & u_{i,\e}^{k}\in L^{2}\left(0,T;H^{1}\left( \Omega_{i,\e}^{k}\right)\right), \quad u_{e}^{\e}\in L^{2}\left(0,T;H^{1}(\Omega_{e,\e})\right),
 \\& \int_{\Omega_{e,\e}} u_{e,\e}(t,x) \ dx=0, \text{ for a.e. } t\in(0,T),
\\ & v_{\e}^{k}\in L^{2}\left( 0,T;H^{1/2}\left(\Gamma_{\e}^{k}\right)\right)\cap L^{r}\left(\Gamma_{\e,T}^{k}\right), \ r\in (2,+\infty)
\\ & s_{\e}\in L^{2}\left(\Gamma_{\e,T}^{1,2}\right), \quad w_{\e}^{k} \in L^{2}(\Gamma_{\e,T}^{k}), \ k=1,2,
\\ & \pt_t v_{\e}^{k}, \ \pt_t w_{\e}^{k} \in L^{2}(\Gamma_{\e,T}^{k}) \text{ for } k=1,2, \quad \pt_t s_{\e} \in L^{2}(\Gamma_{\e,T}^{1,2}).
\end{aligned}
\end{equation*}
\item (Initial conditions).
\begin{equation*}
\begin{aligned}
& v_{\e}^{k}(0,x)=v_{0,\e}^{k}(x), \ w_{\e}^{k}(0,x)=w_{0,\e}^{k}(x) & \text{ a.e. on } \Gamma_{\e,T}^{k}, 
\\ & \text{and } s_{\e}(0,x)=s_{0,\e}(x) & \text{ a.e. on } \Gamma_{\e,T}^{1,2}. 
\end{aligned}
\end{equation*}
\item (Variational equations).
\begin{equation}
\begin{aligned}
&\sum \limits_{k=1,2}\iint_{\Gamma_{\e,T}^{k}} \e\pt_t v_{\e}^{k} \psi_{i}^{k} \ d\sigma_xdt+\iint_{\Gamma_{\e,T}^{1,2}} \frac{\e}{2}\pt_t s_{\e} \Psi \ d\sigma_xdt +\sum \limits_{k=1,2}\int_{\Omega_{i,\e,T}^{k}}\mathrm{M}_{i}^{\e}\nabla u_{i,\e}^{k}\cdot\nabla\varphi_{i}^{k} \ dxdt
\\& \quad +\sum \limits_{k=1,2}\iint_{\Gamma_{\e,T}^{k}} \e\I_{ion}\left(v_{\e}^{k},w_{\e}^{k}\right)\psi_{i}^{k} \ d\sigma_xdt+\frac{1}{2}\iint_{\Gamma_{\e,T}^{1,2}} \e \I_{gap}(s_{\e})\Psi \ d\sigma_xdt
\\& =\sum \limits_{k=1,2}\iint_{\Gamma_{\e,T}^{k}} \e\I_{app,\e}^{k}\psi_{i}^{k} \ d\sigma_xdt
\end{aligned}
\label{Fv_ik_ini_gap}
\end{equation}
\begin{equation}
\begin{aligned}
&\sum \limits_{k=1,2}\iint_{\Gamma_{\e,T}^{k}} \e\pt_t v_{\e}^{k} \psi_{e}^{k} \ d\sigma_xdt-\int_{\Omega_{e,\e,T}}\mathrm{M}_{e}^{\e}\nabla u_{e,\e}\cdot\nabla\varphi_{e} \ dxdt
\\&+\sum \limits_{k=1,2}\iint_{\Gamma_{\e,T}^{k}} \e\I_{ion}\left(v_{\e}^{k},w_{\e}^{k}\right)\psi_{e}^{k} \ d\sigma_xdt =\sum \limits_{k=1,2}\iint_{\Gamma_{\e,T}^{k}} \e\I_{app,\e}^{k}\psi_{e}^{k} \ d\sigma_xdt
\end{aligned}
\label{Fv_e_ini_gap}
\end{equation}
\begin{equation}
\iint_{\Gamma_{\e,T}^{k}} \pt_t w_{\e}^{k}e^{k} \ d\sigma_xdt=\iint_{\Gamma_{\e,T}^{k}} H\left(v_{\e}^{k},w_{\e}^{k}\right) e^{k} \ d\sigma_xdt
\label{Fv_d_ini_gap}
\end{equation}
\end{enumerate}
for all $\varphi_{i}^{k}\in L^{2}\left(0,T;H^{1}\left( \Omega_{i,\e}^{k}\right)\right),$ $\varphi_{e}\in L^{2}\left(0,T;H^{1}(\Omega_{e,\e})\right)$ with 
\begin{itemize}
\item $\psi^{k}=\psi_{i}^{k}-\psi_{e}^{k}:=\left(\varphi_{i}^{k}-\varphi_{e}\right)\vert_{\Gamma_{\e,T}^{k}} \in L^{2}\left( 0,T;H^{1/2}\left(\Gamma_{\e}^{k}\right)\right)\cap L^{r}\left(\Gamma_{\e,T}^{k}\right)$ for $k=1,2,$
\item $\Psi=\Psi_{i}^{1}-\Psi_{i}^{2}:=\left(\varphi_{i}^{1}-\varphi_{i}^{2}\right)\vert_{\Gamma_{\e,T}^{1,2}}\in  L^{2}(\Gamma_{\e,T}^{1,2}),$
\item $e^{k}\in L^{2}(\Gamma_{\e,T}^{k})$ for $k=1,2.$
\end{itemize}
\label{Fv_gap} 

\end{defi}

 
 Then, the existence of the weak solution for the microscopic tridomain problem \eqref{pbscale_gap}-\eqref{cond_ini_vws_gap} is given in the following theorem whose proof is the main issue of the article \cite{BaderTridPart1}, by using the Faedo-Galerkin method.

\begin{thm}[Microscopic Tridomain Model] Assume that the conditions \eqref{A_Mie_gap}-\eqref{A_vw0_gap} hold. Then, System \eqref{pbscale_gap}-\eqref{cond_ini_vws_gap} possesses a unique weak solution in the sense of Definition \ref{Fv_gap} for every fixed $\e>0$.  
 
 Furthermore, this solution verifies the following energy estimates: there exists constants $C_1, C_2, C_3, C_4$ independent of $\e$ such that:
\begin{equation}
\sum\limits_{k=1,2} \norm{\sqrt{\e}v_{\e}^{k}}_{L^{\infty}\left(0,T;L^2(\Gamma_{\e}^{k})\right)}^2+\sum\limits_{k=1,2} \norm{\sqrt{\e}w_{\e}^{k}}_{L^{\infty}\left(0,T;L^2(\Gamma_{\e}^{k})\right)}^2+\norm{\sqrt{\e} s_{\e}}_{L^{\infty}\left(0,T;L^2(\Gamma_{\e}^{1,2})\right)}^2\leq C_1,
\label{E_vw_gap}
\end{equation}

\begin{equation}
\sum\limits_{k=1,2}\norm{u_{i,\e}^{k}}_{L^{2}\left(0,T;H^{1}\left( \Omega_{i,\e}^{k}\right) \right)}+\norm{u_{e}^{\e}}_{L^{2}\left(0,T;H^{1}\left(\Omega_{e,\e}\right) \right)}\leq C_2,
\label{E_u_gap}
\end{equation}

\begin{equation}
\sum\limits_{k=1,2}\norm{\e^{1/r}v_{\e}^{k}}_{L^{r}(\Gamma_{\e,T}^{k})}\leq C_3 \text{ and } \sum\limits_{k=1,2}\norm{\e^{(r-1)/r}\mathrm{I}_{a,ion}(v_{\e}^{k})}_{L^{r/(r-1)}(\Gamma_{\e,T}^{k})}\leq C_4.
\label{E_vr_gap}
\end{equation}
Moreover, if $v_{\e,0}^{k} \in H^{1/2}(\Gamma_{\e}^{k})\cap L^{r}(\Gamma_\e^{k}),$ $k=1,2,$ then there exists a constant $C_5$  independent of $\e$ such that:
\begin{equation}
\sum\limits_{k=1,2} \norm{\sqrt{\e}\pt_t v_{\e}^{k}}_{L^2(\Gamma_{\e,T}^{k})}^2+\sum\limits_{k=1,2} \norm{\sqrt{\e}\pt_t w_{\e}^{k}}_{L^2(\Gamma_{\e,T}^{k})}^2+\norm{\sqrt{\e} \pt_t s_{\e}}_{L^2(\Gamma_{\e,T}^{1,2})}^2\leq C_5.
\label{E_dtv_gap}
\end{equation}

\label{thm_micro_gap}
\end{thm}

 By summing the two first equations in \eqref{Fv_ik_ini_gap}-\eqref{Fv_d_ini_gap} and since $\I_{ion}(v_{\e}^{k},w_{\e}^{k})=\mathrm{I}_{a,ion}(v_{\e}^{k})+\mathrm{I}_{b,ion}(w_{\e}^{k}),$ we can rewrite the weak formulation as follows:
\begin{equation}
\begin{aligned}
&\sum \limits_{k=1,2} \iint_{\Gamma_{\e,T}^{k}}\e \pt_t v_{\e}^{k} \psi^{k} \ d\sigma_xdt+\frac{1}{2}\iint_{\Gamma_{\e,T}^{1,2}}\e \pt_t s_{\e} \Psi \ d\sigma_xdt
\\&+\sum \limits_{k=1,2}\iint_{\Omega_{i,\e,T}^{k}}\mathrm{M}_{i}^{\e}\nabla u_{i,\e}^{k}\cdot\nabla\varphi_{i}^{k} \ dxdt+\iint_{\Omega_{e,\e,T}}\mathrm{M}_{e}^{\e}\nabla u_{e,\e}\cdot\nabla\varphi_{e} \ dxdt
\\&+\sum \limits_{k=1,2}\iint_{\Gamma_{\e,T}^{k}} \e\mathrm{I}_{a,ion}\left(v_{\e}^{k}\right)\psi^{k} \ d\sigma_xdt+\sum \limits_{k=1,2}\iint_{\Gamma_{\e,T}^{k}} \e\mathrm{I}_{b,ion}\left(w_{\e}^{k}\right)\psi^{k} \ d\sigma_xdt
\\&+\frac{1}{2}\iint_{\Gamma_{\e,T}^{1,2}} \e\I_{gap}\left(s_{\e}\right)\Psi \ d\sigma_xdt
=\sum \limits_{k=1,2}\iint_{\Gamma_{\e,T}^{k}} \e\I_{app,\e}^{k}\psi^{k} \ d\sigma_xdt,
\end{aligned}
\label{Fv_ie_gap}
\end{equation}
\begin{equation}
\iint_{\Gamma_{\e,T}^{k}} \pt_t w_{\e}^{k}e^{k} \ d\sigma_xdt=\iint_{\Gamma_{\e,T}^{k}} H\left(v_{\e}^{k},w_{\e}^{k}\right) e^{k} \ d\sigma_xdt.
\label{Fv_d_gap}
\end{equation}

\section{Unfolding Homogenization Method}\label{methodunf_gap}

Our derivation of the tridomain model is based on a new approach describing not only the electrical activity but also the effect of the cell membrane and gap junctions in the heart tissue. 
Our goal in this section is to describe the asymptotic behavior, as $\e \rightarrow 0$, of the solution $(u_{i,\e}^{1}, u_{i,\e}^{2}, u_{e,\e}, w_{\e}^{1}, w_{\e}^{2})$ given by System \eqref{pbscale_gap}-\eqref{cond_ini_vws_gap}. We do this by following a three-steps procedure: In Step \ref{unfolded formulation}, the weak formulation of the microscopic tridomain model \eqref{pbscale_gap}-\eqref{cond_ini_vws_gap} is written by another one, called "unfolded" formulation, based on the unfolding operators stated in the previous part. As Step \ref{convergence unfolded formulation}, we can pass to the limit as $\e\rightarrow 0$ in the unfolded formulation using some a priori estimates and compactness argument to get the corresponding homogenization equation. In Step \ref{macro_gap}, we take a special form of test functions to obtain finally the macroscopic tridomain model.
 
\subsection{Unfolded formulation of the microscopic tridomain model}\label{unfolded formulation}{} Based on the properties of the unfolding operators, we rewrite the weak formulation \eqref{Fv_ie_gap}-\eqref{Fv_d_gap} in the "unfolded" form. First, we denote by $E_i$ with $i = 1,\dots, 5$ the terms of the equation \eqref{Fv_ie_gap} which is rewritten as follows (to respect the order):
 \begin{equation*}
E_1+E_2+E_3+E_4+E_5+E_6+E_7=E_8.
 \end{equation*}
Using property \eqref{P_uo4_gap} of Proposition \ref{prop_uo_gap}, then the first and second term of \eqref{Fv_ie_gap}  is rewritten as follows:
\begin{align*}
E_1 
&=\sum \limits_{k=1,2}\iint_{\widehat{\Gamma}_{\e,T}^{k}} \e \pt_t v_{\e}^{k} \psi^{k} \ d\sigma_xdt+\sum \limits_{k=1,2}\iint_{\Gamma_{\e,T}^{k}\cap\Lambda_{\e,T}} \e \pt_t v_{\e}^{k} \psi^{k} \ d\sigma_xdt
\\&=\dfrac{1}{\abs{Y}}\sum \limits_{k=1,2}\iiint_{\Omega_{T}\times \Gamma^{k}}\T_{\e}^{b,k}(\pt_t v_{\e}^{k})\T_{\e}^{b,k}(\psi^{k})\ dxd\sigma_ydt+\sum \limits_{k=1,2}\iint_{\Gamma_{\e,T}^{k}\cap\Lambda_{\e,T}} \e\pt_t v_{\e}^{k}\psi^{k} \ d\sigma_xdt
\\& :=J_1+R_1.\\
E_2 
&=\frac{1}{2}\iint_{\widehat{\Gamma}_{\e,T}^{1,2}} \e \pt_t s_{\e} \Psi\ d\sigma_xdt+\frac{1}{2}\iint_{\Gamma_{\e,T}^{1,2}\cap\Lambda_{\e,T}} \e \pt_t s_{\e} \Psi \ d\sigma_xdt
\\&=\dfrac{1}{2\abs{Y}}\iiint_{\Omega_{T}\times \Gamma^{1,2}}\T_{\e}^{b,1,2}(\pt_t s_{\e})\T_{\e}^{b,1,2}(\Psi)\ dxd\sigma_ydt+\frac{1}{2}\iint_{\Gamma_{\e,T}^{1,2}\cap\Lambda_{\e,T}} \e\pt_t s_{\e}\Psi \ d\sigma_xdt
\\& :=J_2+R_2.
\end{align*}

Similarly, we rewrite the third and fourth term using the property \eqref{P_uo3_gap} of Proposition \ref{prop_uo_gap}:
\begin{align*}
E_3 &=\dfrac{1}{\abs{Y}}\sum \limits_{k=1,2}\iiint_{\Omega_{T}\times Y_{i}^{k}}\T_{\e}^{i,k}(\mathrm{M}_{i}^{\e}) \T_{\e}^{i,k}(\nabla u_{i,\e}^{k})\T_{\e}^{i,k}(\nabla \varphi_i^{k}) \ dxdydt
\\& \quad +\sum \limits_{k=1,2}\iint_{\Lambda_{i,\e,T}^{k}}\mathrm{M}_{i}^{\e}\nabla u_{i,\e}^{k}\cdot\nabla\varphi_i^{k} \ dxdt
\\&:=J_3+R_3 \\
E_4 &=\dfrac{1}{\abs{Y}}\iiint_{\Omega_{T}\times Y_{e}}\T_{\e}^{e}(\mathrm{M}_{e}^{\e}) \T_{\e}^{e}(\nabla u_{e,\e})\T_{\e}^{e}(\nabla \varphi_e) \ dxdydt
\\& \quad +\iint_{\Lambda_{e,\e,T}}\mathrm{M}_{e}^{\e}\nabla u_{e,\e}\cdot\nabla\varphi_e \ dxdt
\\&:=J_4+R_4
\end{align*}

Due to the form of $\mathrm{I}_{\ell,ion},$ we use the property \eqref{P_uo2_gap}-\eqref{P_uo4_gap} of Proposition \ref{prop_uo_gap} to obtain $\T_{\e}^{b,k}\left(\mathrm{I}_{\ell,ion}(\cdot)\right)=\mathrm{I}_{\ell,ion}\left( \T_{\e}^{b,k}(\cdot)\right)$ for $\ell=a,b$ and $k=1,2.$ Thus, we arrive to: 
\begin{align*}
E_5
&=\dfrac{1}{\abs{Y}}\sum \limits_{k=1,2}\iiint_{\Omega_{T}\times \Gamma^{k}}\T_{\e}^{b,k}\left(\mathrm{I}_{a,ion}(v_{\e}^{k})\right) \T_{\e}^{b,k}(\psi^{k})\ dxd\sigma_ydt+\sum \limits_{k=1,2}\iint_{\Gamma_{\e,T}^{k}\cap\Lambda_{\e,T}} \e\mathrm{I}_{a,ion}(v_{\e}^{k})\psi^{k} \ d\sigma_xdt
\\&=\dfrac{1}{\abs{Y}}\sum \limits_{k=1,2}\iiint_{\Omega_{T}\times \Gamma^{k}}\mathrm{I}_{a,ion}\left( \T_{\e}^{b,k}(v_{\e}^{k})\right) \T_{\e}^{b,k}(\psi^{k})\ dxd\sigma_ydt+\sum \limits_{k=1,2}\iint_{\Gamma_{\e,T}^{k}\cap\Lambda_{\e,T}} \e\mathrm{I}_{a,ion}(v_{\e}^{k})\psi^{k} \ d\sigma_xdt
\\& :=J_5+R_5\\
E_6
& =\dfrac{1}{\abs{Y}}\sum \limits_{k=1,2}\iiint_{\Omega_{T}\times \Gamma^{k}}\T_{\e}^{b,k}(\mathrm{I}_{b,ion}(w_{\e}^{k}))\T_{\e}^{b,k}(\psi^{k})\ dxd\sigma_ydt+\sum \limits_{k=1,2}\iint_{\Gamma_{\e,T}^{k}\cap\Lambda_{\e,T}} \e\mathrm{I}_{b,ion}(w_{\e}^{k})\psi^{k} \ d\sigma_xdt
\\& =\dfrac{1}{\abs{Y}}\sum \limits_{k=1,2}\iiint_{\Omega_{T}\times \Gamma^{k}}\mathrm{I}_{b,ion}\left( \T_{\e}^{b,k}(w_{\e}^{k})\right) \T_{\e}^{b,k}(\psi^{k})\ dxd\sigma_ydt+\sum \limits_{k=1,2}\iint_{\Gamma_{\e,T}^{k}\cap\Lambda_{\e,T}} \e\mathrm{I}_{b,ion}(w_{\e}^{k})\psi^{k} \ d\sigma_xdt
\\& :=J_6+R_6
\end{align*}
Similarly, we can rewrite the last two terms of \eqref{Fv_ie_gap} by taking account the form of $\I_{gap}$ as follows:
\begin{align*}
E_7
&=\dfrac{1}{2\abs{Y}}\iiint_{\Omega_{T}\times \Gamma^{1,2}}\I_{gap}\left(\T_{\e}^{b,1,2}(s_{\e})\right) \T_{\e}^{b,1,2}(\Psi)\ dxd\sigma_ydt+\frac{1}{2}\iint_{\Gamma_{\e,T}^{1,2}\cap\Lambda_{\e,T}} \e\I_{gap}(s_{\e})\Psi \ d\sigma_xdt
\\& :=J_7+R_7
\end{align*} 
\begin{align*}
E_8
&=\dfrac{1}{\abs{Y}}\sum \limits_{k=1,2}\iiint_{\Omega_{T}\times \Gamma^{k}}\T_{\e}^{b,k}(\I_{app,\e}^{k})\T_{\e}^{b,k}(\psi^{k})\ dxd\sigma_ydt+\sum \limits_{k=1,2}\iint_{\Gamma_{\e,T}^{k}\cap\Lambda_{\e,T}} \e\I_{app,\e}^{k}\psi^{k} \ d\sigma_xdt
\\& :=J_8+R_8
\end{align*} 

 Collecting the previous estimates, we readily obtain from \eqref{Fv_ie_gap} the following "unfolded" formulation:
  \begin{equation}
\begin{aligned}
& \dfrac{1}{\abs{Y}}\sum \limits_{k=1,2}\iiint_{\Omega_{T}\times \Gamma^{k}}\T_{\e}^{b,k}(\pt_t v_{\e}^{k})\T_{\e}^{b,k}(\psi^{k})\ dxd\sigma_ydt+\dfrac{1}{2\abs{Y}}\iiint_{\Omega_{T}\times \Gamma^{1,2}}\T_{\e}^{b,1,2}(\pt_t s_{\e})\T_{\e}^{b,1,2}(\Psi)\ dxd\sigma_ydt
\\ & \quad +\dfrac{1}{\abs{Y}}\sum \limits_{k=1,2}\iiint_{\Omega_{T}\times Y_{i}^{k}}\T_{\e}^{i,k}(\mathrm{M}_{i}^{\e}) \T_{\e}^{i,k}(\nabla u_{i,\e}^{k})\T_{\e}^{i,k}(\nabla \varphi_i^{k}) \ dxdydt
\\& \quad +\dfrac{1}{\abs{Y}}\iiint_{\Omega_{T}\times Y_{e}}\T_{\e}^{e}(\mathrm{M}_{e}^{\e}) \T_{\e}^{e}(\nabla u_{e,\e})\T_{\e}^{e}(\nabla \varphi_e) \ dxdydt
\\& \quad +\dfrac{1}{\abs{Y}}\sum \limits_{k=1,2}\iiint_{\Omega_{T}\times \Gamma^{k}}\mathrm{I}_{a,ion}\left( \T_{\e}^{b,k}(v_{\e}^{k})\right) \T_{\e}^{b,k}(\psi^{k})\ dxd\sigma_ydt
\\& \quad +\dfrac{1}{\abs{Y}}\sum \limits_{k=1,2}\iiint_{\Omega_{T}\times \Gamma^{k}}\mathrm{I}_{b,ion}\left( \T_{\e}^{b,k}(w_{\e}^{k})\right) \T_{\e}^{b,k}(\psi^{k})\ dxd\sigma_ydt
\\& \quad +\dfrac{1}{2\abs{Y}}\iiint_{\Omega_{T}\times \Gamma^{1,2}}\I_{gap}\left(\T_{\e}^{b,1,2}(s_{\e})\right) \T_{\e}^{b,1,2}(\Psi)\ dxd\sigma_ydt
\\&=\dfrac{1}{\abs{Y}}\sum \limits_{k=1,2}\iiint_{\Omega_{T}\times \Gamma^{k}}\T_{\e}^{b,k}(\I_{app,\e}^{k})\T_{\e}^{b,k}(\psi^{k})\ dxd\sigma_ydt
\\ & \quad +R_8-R_7-R_6-R_5-R_4-R_3-R_2-R_1
\end{aligned}
\label{Fv_ie_unf_gap}
\end{equation}

 Similarly, the "unfolded" formulation of \eqref{Fv_d_gap} is given by:
 \begin{equation}
\begin{aligned}
&\dfrac{1}{\abs{Y}}\iiint_{\Omega_{T}\times \Gamma^{k}}\T_{\e}^{b,k}(\pt_t w_{\e}^{k})\T_{\e}^{b,k}(e^{k})\ dxd\sigma_ydt
\\ & -\dfrac{1}{\abs{Y}}\iiint_{\Omega_{T}\times \Gamma^{k}}H(\T_{\e}^{b,k}(v_{\e}^{k}),\T_{\e}^{b,k}(w_{\e}^{k}))\T_{\e}^{b,k}(e^{k})\ dxd\sigma_ydt
\\& =-\e\iint_{\Gamma_{\e,T}^{k}\cap\Lambda_{\e,T}} \pt_t w_{\e}^{k}e^{k} \ d\sigma_xdt+\e \iint_{\Gamma_{\e,T}^{k}\cap\Lambda_{\e,T}}  H(v_{\e}^{k}, w_{\e}^{k}) e^{k} \ d\sigma_xdt
\\&:=R_9+R_{10}
\end{aligned}
\label{Fv_d_unf_gap}
\end{equation}

\subsection{Convergence of the unfolded formulation}\label{convergence unfolded formulation}
In this part, we pass to the limit in \eqref{Fv_ie_unf_gap}-\eqref{Fv_d_unf_gap}. First, we prove that:

\begin{equation*}
R_1,\cdots,R_{10} \underset{\e \rightarrow 0}{\longrightarrow} 0,
\end{equation*}
by making use of estimates \eqref{E_vw_gap}-\eqref{E_dtv_gap}. So, we prove that $R_3\rightarrow 0$ when $\e\rightarrow 0$ and the proof for the other terms is similar. 
 First, by  Cauchy-Schwarz inequality,
one has 
\begin{equation*}
\begin{aligned}
R_3=\sum \limits_{k=1,2}\iint_{\Lambda_{i,\e,T}^{k}}\mathrm{M}_{i}^{\e}\nabla u_{i,\e}^{k}\cdot\nabla\varphi_i^{k} \ dxdt\leq\sum \limits_{k=1,2}\norm{\mathrm{M}_{i}^{\e}\nabla u_{i,\e}^{k}}_{L^2\left(\Omega_{i,\e,T}^{k}\right)} \left(\iint_{\Lambda_{i,\e,T}^{k}} \abs{\nabla\varphi_i^{k}}^2 \ dxdt\right)^{1/2}. 
\end{aligned}
\end{equation*}
In addition, we observe that $\abs{\Lambda_{i,\e}^{k}}\rightarrow 0$ and $\nabla\varphi_i^{k} \in L^2(\Omega_{i,\e}^{k}).$ Consequently, by Lebesgue dominated convergence theorem, one gets for $k=1,2:$ 
\begin{equation*}
\iint_{\Lambda_{i,\e}^{k}} \abs{\nabla\varphi_i^{k}}^2\rightarrow 0, \text{ as } \e\rightarrow 0.
\end{equation*} 
Finally, by using Hölder's inequality, the result follows by making use of estimate \eqref{E_u_gap} and assumption \eqref{A_Mie_gap} on $\mathrm{M}_{i}^{\e}$. 
   
 Let us now elaborate the convergence results of $J_1,\cdots,J_8$. Using property \eqref{P_uo5_gap} of Proposition \ref{prop_uo_gap} and due to the regularity of test functions, we know that the following strong convergence hold:
 \begin{align*}
& \T_{\e}^{b,k}(\psi^{k}) \rightarrow \psi^{k} \text{ and } \T_{\e}^{b,k}(e^{k}) \rightarrow e^{k} \text{ strongly in } L^{2}(\Omega_T\times \Gamma^{k})
\\& \T_{\e}^{b,1,2}(\Psi) \rightarrow \Psi \text{ strongly in } L^{2}(\Omega_T\times \Gamma^{1,2}) 
 \end{align*}
and
\begin{align*}
 \T_{\e}^{i,k}(\varphi_{i}^{k})  \rightarrow \varphi_{i}^{k} &\text{ strongly in } L^{2}(\Omega_T\times Y_{i}^{k}),
\\ \T_{\e}^{e}(\varphi_e)  \rightarrow \varphi_e &\text{ strongly in } L^{2}(\Omega_T\times Y_{e}).
\end{align*}
Next, we want to use the a priori estimates \eqref{E_vw_gap}-\eqref{E_dtv_gap} to verify that the remaining terms of the equations in the unfolded formulation \eqref{Fv_ie_unf_gap}-\eqref{Fv_d_unf_gap} are weakly convergent. Using estimation \eqref{E_u_gap}, we deduce that there exist $u_{i}^{k},u_{e} \in L^{2}\left( 0,T; H^{1}(\Omega)\right),$ $\widehat{u}_{i}^{k} \in L^{2}\left( 0,T; L^{2}\left(\Omega, H_{\#}^{1}(Y_{i}^{k})\right)\right)$ for $k=1,2$ and $\widehat{u}_{e} \in L^{2}\left( 0,T; L^{2}\left(\Omega, H_{\#}^{1}(Y_{e})\right)\right)$  such that, up to a subsequence (see for instance Theorem 3.12 in \cite{doinaunf12}), the following convergences hold as $\e$ goes to zero:
\begin{align*}
&\T_{\e}^{i,k}(u_{i,\e}^{k})\rightharpoonup u_{i}^{k} \text{ weakly in } L^{2}\left( 0,T; L^{2}\left(\Omega \times Y_{i}^{k}\right)\right),
\\ & \T_{\e}^{i,k}(\nabla u_{i,\e}^{k}) \rightharpoonup \nabla u_{i}^{k}+\nabla_y \widehat{u}_{i}^{k} \text{ weakly in } L^{2}(\Omega_T\times Y_{i}^{k}),
\end{align*}
and
\begin{align*}
&\T_{\e}^{e}(u_{e,\e})\rightharpoonup u_e \text{ weakly in } L^{2}\left( 0,T; L^{2}\left(\Omega \times Y_{e}\right)\right),
\\ & \T_{\e}^{e}(\nabla u_{e,\e}) \rightharpoonup \nabla u_{e}+\nabla_y \widehat{u}_{e} \text{ weakly in } L^{2}(\Omega_T\times Y_{i}^{k}),
\end{align*}
with the space $H_{\#}^{1}$ given by \eqref{W_gap}.
Thus, since $\T_{\e}^{i,k}\left(\mathrm{M}_{i}^{\e}\right)\rightarrow \mathrm{M}_{i}$ a.e. in $\Omega\times Y_{i}^{k}$ for $k=1,2$ and $\T_{\e}^{e}\left(\mathrm{M}_{e}^{\e}\right) \rightarrow \mathrm{M}_{e}$ a.e. in $\Omega\times Y_{e},$ one obtains:
\begin{align*}
&J_3\underset{\e \rightarrow 0}{\longrightarrow}\dfrac{1}{\abs{Y}}\sum \limits_{k=1,2}\iiint_{\Omega_{T}\times Y_{i}^{k}}\mathrm{M}_{i} \left[ \nabla u_{i}^{k}+\nabla_y \widehat{u}_{i}^{k}\right] \nabla \varphi_{i}^{k} \ dxdydt,
\\&J_4\underset{\e \rightarrow 0}{\longrightarrow} \dfrac{1}{\abs{Y}}\iiint_{\Omega_{T}\times Y_{e}}\mathrm{M}_{e} \left[ \nabla u_{e}+\nabla_y \widehat{u}_{e}\right] \nabla \varphi_e \ dxdydt.
\end{align*}

 Furthermore, we need to establish the weak convergence of the unfolded sequences that corresponds to $v_{\e}^{k}, w_{\e}^{k}, s_\e$ and $\I_{app,\e}^{k}$ for $k=1,2.$ In order to establish the convergence of $\T_{\e}^{b,k}(\pt_t v_{\e}^{k}),$ we use estimation \eqref{E_dtv_gap} to get for $k=1,2$
 $$ \norm{\T_{\e}^{b,k}(\pt_t v_{\e}^{k})}_{L^{2}(\Omega_T \times \Gamma^{k})}\leq\e^{1/2}\abs{Y}^{1/2}\norm{\pt_t v_{\e}^{k}}_{L^{2}(\Gamma_{\e,T}^{k})}\leq C.$$
So there exists $V^{k} \in L^{2}\left(\Omega_T\right)$ such that $\T_{\e}^{b,k}(\pt_t v_{\e}^{k})\rightharpoonup V^{k}$  weakly in $L^{2}(\Omega_T \times \Gamma^{k})$ with $k=1,2.$ By a classical integration argument, one can show that $V^{k}=\pt_t v^{k}.$ Therefore, we deduce that
$$\T_{\e}^{b,k}(\pt_t v_{\e}^{k})\rightharpoonup \pt_t v^{k} \text{ weakly in } L^{2}(\Omega_T \times \Gamma^{k}).$$
Thus, we obtain
\begin{align*}
J_1&=\dfrac{1}{\abs{Y}}\sum \limits_{k=1,2}\iiint_{\Omega_{T}\times \Gamma^{k}} \T_{\e}^{b,k}(\pt_t v_{\e}^{k}) \T_{\e}^{b,k}(\psi^{k})\ dxd\sigma_ydt 
\\& \underset{\e \rightarrow 0}{\longrightarrow} \dfrac{1}{\abs{Y}}\sum \limits_{k=1,2}\iiint_{\Omega_{T}\times \Gamma^{k}}\pt_t v^{k} \psi^{k} \ dxd\sigma_ydt.
\end{align*}

By the same strategy for the convergence of $J_1,$ there exits $S \in L^{2}\left(\Omega_T\right)$ such that $\T_{\e}^{b,1,2}(\pt_t s_{\e})\rightharpoonup S$  weakly in $L^{2}(\Omega_T \times \Gamma^{1,2}).$ Similarly, we get $S=\pt_t s.$ Thus, one has 
\begin{align*}
& J_2=\dfrac{1}{2\abs{Y}}\iiint_{\Omega_{T}\times \Gamma^{1,2}} \T_{\e}^{b,1,2}(\pt_t s_{\e}) \T_{\e}^{b,1,2}(\Psi)\ dxd\sigma_ydt
\\ & \underset{\e \rightarrow 0}{\longrightarrow} \dfrac{1}{2\abs{Y}}\iiint_{\Omega_{T}\times \Gamma^{1,2}}\pt_t s \Psi \ dxd\sigma_ydt.
\end{align*}
Now, making use of estimate \eqref{E_vw_gap} with property \eqref{P_uo4_gap} of Proposition \ref{prop_uo_gap}, one has
\begin{align*}
&\norm{\T_{\e}^{b,k}(w_{\e}^{k})}_{L^{2}(\Omega_T \times \Gamma^{k})}\leq\e^{1/2}\abs{Y}^{1/2}\norm{w_{\e}^{k}}_{L^{2}(\Gamma_{\e,T}^{k})}\leq C,
\\& \norm{\T_{\e}^{b,1,2}(s_{\e})}_{L^{2}(\Omega_T \times \Gamma^{1,2})}\leq\e^{1/2}\abs{Y}^{1/2}\norm{s_{\e}}_{L^{2}(\Gamma_{\e,T}^{1,2})}\leq C.
\end{align*}
Then, up to a subsequences,
\begin{align*}
&\T_{\e}^{b,k}(w_{\e}^{k})\rightharpoonup w^{k} \text{ weakly in } L^{2}(\Omega_T \times \Gamma^{k}),
\\&\T_{\e}^{b,1,2}(s_{\e})\rightharpoonup s \text{ weakly in } L^{2}(\Omega_T \times \Gamma^{1,2}).
\end{align*}
So, by linearity of $\mathrm{I}_{b,ion}$ and of $\I_{gap}$ we have respectively:
\begin{align*}
&J_6=\dfrac{1}{\abs{Y}}\sum \limits_{k=1,2}\iiint_{\Omega_{T}\times \Gamma^{k}}\mathrm{I}_{b,ion}\left( \T_{\e}^{b,k}(w_{\e}^{k})\right) \T_{\e}^{b,k}(\psi^{k})\ dxd\sigma_ydt
\\ & \qquad \underset{\e \rightarrow 0}{\longrightarrow} \dfrac{1}{\abs{Y}}\sum \limits_{k=1,2}\iiint_{\Omega_{T}\times \Gamma^{k}}\mathrm{I}_{b,ion}(w^{k}) \psi^{k} \ dxd\sigma_ydt,
\\& J_7=\dfrac{1}{2\abs{Y}}\iiint_{\Omega_{T}\times \Gamma^{1,2}}\mathrm{I}_{gap}\left( \T_{\e}^{b,1,2}(s_{\e})\right) \T_{\e}^{b,1,2}(\Psi)\ dxd\sigma_ydt
\\ & \qquad \underset{\e \rightarrow 0}{\longrightarrow} \dfrac{1}{2\abs{Y}}\iiint_{\Omega_{T}\times \Gamma^{1,2}}\I_{gap}(s) \Psi \ dxd\sigma_ydt.
\end{align*}
Similarly, exploiting assumption \eqref{A_iapp_gap} on $\I_{app,\e}^{k}$, we obtain the following  convergence:
\begin{align*}
J_8&=\dfrac{1}{\abs{Y}}\sum \limits_{k=1,2}\iiint_{\Omega_{T}\times \Gamma^{k}}\T_{\e}^{b,k}(\I_{app,\e}^{k})\T_{\e}^{b,k}(\psi^{k})\ dxd\sigma_ydt
\\ & \underset{\e \rightarrow 0}{\longrightarrow} \dfrac{1}{\abs{Y}}\sum \limits_{k=1,2}\iiint_{\Omega_{T}\times \Gamma^{k}}\I_{app}^{k}\psi^{k}\ dxd\sigma_ydt.
\end{align*}

\begin{rem} Proceeding exactly as in \cite{BaderUnf}, we prove that the limits $v^{k}$ and $s$ coincide respectively with $u_i^{k}-u_e$ for $k=1,2$ and $u_i^{1}-u_i^{2}.$ Furthermore, since we have assumed that the initial data $v_{0,\e}^{k},w_{0,\e}^{k}$ for $k=1,2$ and $s_{0,\e}$ introduced in \eqref{cond_ini_vws_gap}, are also uniformly bounded in the adequate norm $($see assumption \eqref{A_vw0_gap}$)$. Then, using the weak formulation \eqref{Fv_ie_unf_gap}-\eqref{Fv_d_unf_gap}, we prove similarly that $v^{k}(0,x)=v_{0}^{k}(x)$ a.e. on $\Omega,$ since, by construction, $v_{\e}^{k}(0,x)=v_{0,\e}^{k}(x)$ a.e. on $\Gamma_{\e}^{k}$ for $k=1,2$. The same argument holds for the initial condition of $w_{\e}^{k}$ for $k=1,2$ and of $s_{\e}$.
\end{rem}

 It remains to obtain the limit of $J_5$ containing the ionic function $\mathrm{I}_{a,ion}.$ By the regularity of $\psi^{k}$, it sufficient to show the weak convergence of $\mathrm{I}_{a,ion}\left( \T_{\e}^{b,k}(v_{\e}^{k})\right) $ to $\mathrm{I}_{a,ion}(v^{k})$ in $L^2(\Omega_T\times
\Gamma^{k}).$  Due to the non-linearity of $\mathrm{I}_{a,ion},$ the weak convergence will not be enough. It is difficult to pass to the limit of this term on the microscopic membrane surface. 
 Therefore, we need the strong convergence of $\T_{\e}^{b,k}(v_{\e}^{k})$ to $v^{k}$ in $L^2(\Omega_T\times \Gamma^{k})$ for $k=1,2$ that we obtain by using Kolmogorov-Riesz type compactness criterion that can be found as Corollary 2.5 in \cite{maria}:
 \begin{prop}[Kolmogorov-Riesz type compactness result] Let $\Omega \subset \R^d$ be an open and bounded set. Let $F \subset L^{p}(\Omega,B)$ for a Banach space B and $p\in[1;+\infty).$ For $f\in F$ and $\xi\in\R^d,$ we define $\tau_{\xi}f(x):=f(x+\xi).$ Then $F$ is relatively compact in $L^{p}(\Omega,B)$ if and only if
\begin{itemize}
\item[$(i)$] for every measurable set $ A\subset \Omega$ the set $\lbrace \int_A f dx \ 
: \ f \in F \rbrace$ is relatively compact in $B,$
\item[$(ii)$] for all $\lambda>0,$ $\xi\in \R^d$ and $\xi_i\geq 0,$ $i=1,\dots,d,$ there holds
$$\underset{f\in F}{\sup} \norm{\tau_{\xi}f-f}_{L^p\left(\Omega_\lambda^\xi,B\right)}\rightarrow 0, \text{ for } h\rightarrow 0, $$
where $\Omega_\lambda^\xi:=\lbrace x\in \Omega_\lambda : x+\xi \in \Omega_\lambda\rbrace$ and $\Omega_\lambda:=\lbrace x\in \Omega : dist(x,\pt \Omega)>\lambda\rbrace,$
\item[$(iii)$] for $\lambda>0,$ there holds $\underset{f\in F}{\sup} \int_{\Omega\setminus\Omega_\lambda} \abs{f(x)}^p dx\rightarrow 0$ for $\lambda \rightarrow 0.$
\end{itemize}
\label{kolmo_gap}
\end{prop}
 
To cope with this, in the following lemma, we derive the convergence of the nonlinear term $\mathrm{I}_{a,ion}:$
\begin{lem}
The following convergence holds for $k=1,2$:
\begin{equation*}
  \T_{\e}^{b,k}(v_{\e}^{k})\rightarrow v^{k} \text{ strongly in } L^2(\Omega_T\times \Gamma^{k}),
\end{equation*}
 as $\e\rightarrow 0.$ Moreover, we have for $k=1,2$:
\begin{equation*}
\mathrm{I}_{a,ion}\left( \T_{\e}^{b,k}(v_{\e}^{k})\right)\rightarrow \mathrm{I}_{a,ion}(v^{k})\text{ strongly in } L^q(\Omega_T\times \Gamma^{k}), \ \forall q\in[1,r/(r-1)),
\end{equation*}
as $\e\rightarrow 0.$
\end{lem}
\begin{proof}
We follow the same idea to the proof of Lemma 5.3 in \cite{bendunf19}.
The proof of the first convergence is based on the Kolmogorov compactness criterion \ref{kolmo_gap}. So, we want to verify that the sequence $\lbrace\T_{\e}^{b,k}(v_{\e}^{k})\rbrace_{\e>0}$  of unfolded membrane potentials satisfies the assumptions of Proposition \ref{kolmo_gap} with $B=L^2\left( 0,T;L^{2}(\Gamma^{k})\right)$ for $k=1,2$ and $p=2$. It is carried out by proving three conditions:

 $\textbf{(i)}$ Let $A\subset\Omega$ a measurable set. We define the sequence $\lbrace v_{A,\e}^{k} \rbrace_{\e>0}$ as follows:
\begin{equation*}
v_{A,\e}^{k}(t,y):=\int_{A} \T_{\e}^{b,k}(v_{\e}^{k})(t,x,y) \ dx, \text{ for a.e. } (t,y)\in (0,T) \times \Gamma^{k} \ (k=1,2).
\end{equation*} 
 It remains to show that the sequence $v_{A,\e}^{k} \in L^2\left( 0,T;H^{1/2}(\Gamma^{k})\right)$ is relatively compact in the space $L^2\left( 0,T;L^{2}(\Gamma^{k})\right)$ for $k=1,2$. Since the embedding $H^{1/2}(\Gamma^{k}) \hookrightarrow L^{2}(\Gamma^{k})$ is compact, we have to show that the sequence $v_{A,\e}^{k}$ is bounded in $L^2\left( 0,T;H^{1/2}(\Gamma^{k})\right) \cap H^1\left( 0,T;L^{2}(\Gamma^{k})\right)$ with $k=1,2.$\\
We first observe that for $k=1,2$
 \begin{equation*}
 \begin{aligned}
 \norm{v_{A,\e}^{k}}_{H^{1/2}(\Gamma^{k})}^2 
 & = \displaystyle \int_{\Gamma^{k}}\card{\int_{A} \T_{\e}^{b,k}(v_{\e}^{k})(t,x,y) \ dx}^2 d\sigma_y
 \\ & \quad +\displaystyle \iint_{\Gamma^{k}\times \Gamma^{k}} \int_{A} \dfrac{\card{ \T_{\e}^{b,k}(v_{\e}^{k})(t,x,y_1)-\T_{\e}^{b,k}(v_{\e}^{k})(t,x,y_2)}^2}{\card{ y_1-y_2 }^{d+1}} \ dxd\sigma_{y_1}d\sigma_{y_2}
 \\& :=\norm{v_{A,\e}^{k}}_{L^{2}(\Gamma^{k})}^2+ \norm{v_{A,\e}^{k}}_{H^{1/2}_0(\Gamma^{k})}^2.
 \end{aligned}
 \end{equation*}
In view of Fubini theorem, Cauchy-Schwarz inequality and estimate \eqref{E_vw_gap}, it follows that for $k=1,2$
 \begin{align*}
 \norm{v_{A,\e}^{k}}_{L^{2}(\Gamma^{k}_T)}^2 
 &\leq C \displaystyle \int_{0}^{T}\int_{\Omega}\int_{\Gamma^{k}}\card{ \T_{\e}^{b,k}(v_{\e}^{k})(t,x,y) }^2 d\sigma_ydxdt
 \\& \leq C \norm{\sqrt{\e}v_{\e}^{k}}_{L^{2}(\Gamma_{\e,T}^{k})}^2 \leq C.
 \end{align*}
 Next, we only need to bound the $H^{1/2}_0$ semi-norm and this is done as follows. Since $v_\e=\left( u_{i}^{\e}-u_{e}^{\e}\right){\vert \Gamma_{\e}^{k}}$ for $k=1,2,$ we use again Fubini theorem and Jensen inequality together with the trace inequality in Remark \ref{trace_ineq_gap} to obtain
 \begin{align*}
\norm{v_{A,\e}^{k}}_{H^{1/2}_0(\Gamma^{k})}^2
&\leq C \left[\int_{\Omega}\norm{\T_{\e}^{b,k}(v_{\e}^{k})}_{H^{1/2}_0(\Gamma^{y})}^2dxdt\right] 
\\ & \leq C\left[ \norm{u_{i,\e}^{k}}_{L^{2}(\Omega_{i,\e}^{k})}^2+\e^2\norm{\nabla u_{i,\e}^{k}}_{L^{2}(\Omega_{i,\e}^{k})}^2+\norm{u_{e,\e}}_{L^{2}(\Omega_{e,\e})}^2+\e^2\norm{\nabla u_{e,\e}}_{L^{2}(\Omega_{e,\e})}^2\right].
 \end{align*}
 Hence, integrating over $(0,T)$ and using the a priori estimates \eqref{E_u_gap}, we have showed that the sequence $v_{A,\e}^{k}$ is bounded in $L^2\left( 0,T;H^{1/2}(\Gamma^{k})\right)$ for $k=1,2.$\\
By a similar argument and making use of the estimate \eqref{E_dtv_gap} on $\e^{1/2}\pt_t v_\e^{k}$, we can also show that
\begin{equation*}
\norm{\pt_t v_{A,\e}^{k}}_{L^{2}(\Gamma^{k}_T)} \leq C, \text{ with } k=1,2.
\end{equation*}
Finally, we deduce that the sequence $v_{A,\e}^{k}$ is bounded in $L^2\left( 0,T;H^{1/2}(\Gamma^{k})\right) \cap H^1\left( 0,T;L^{2}(\Gamma^{k})\right)$ and due to the Aubin-Lions Lemma the sequence is relatively compact in $L^2\left( 0,T;L^{2}(\Gamma^{k})\right)$ with $k=1,2.$
 
 $\textbf{(ii)}$ Due to the decomposition of the domain given in Subsection \ref{unfop_gap}, $\Omega$ can always be represented by a union of scaled and translated reference cells. Fix $\e>0$ and let $k \in \Xi_\e,$ be an index set such that
\begin{equation*}
\widehat{\Omega}^{\e}= \underset{h \in \Xi_\e }{\bigcup} \e (h_{\ell} + Y), \text{ with } h_\ell:=( h_1\ell^\text{mes}_1,\dots,  h_d \ell^\text{mes}_d ).
\end{equation*}
Note that $x \in \e (h_{\ell} +Y) \Leftrightarrow \left[ \dfrac{x}{\e} \right]_Y=h_{\ell}.$ For every fixed $h \in \Xi_\e,$ we subdivide the cell $\e (h_{\ell} +Y)$ into subsets $\e \left(h_{\ell} +Y\right)^\sigma$ with $\sigma\in\left\lbrace 0,1 \right\rbrace^d,$ defined as follows
\begin{equation*}
\e (k_{\ell} +Y)^\sigma :=\left\lbrace x \in \e (k_{\ell} +Y)  : \e\left[\dfrac{x+\e\left\lbrace  \dfrac{\xi}{\e} \right\rbrace_Y}{\e}\right]_Y =\e (h_{\ell} +\sigma) \right\rbrace,
\end{equation*}
for a given $\xi \in \R^d.$ It holds $\e (h_{\ell} +Y)=\underset{\sigma\in\left\lbrace 0,1 \right\rbrace^d }{\bigcup}\e (h_{\ell} +Y)^\sigma.$\\
We use the same notation as in Proposition \ref{kolmo_gap}. Now, we compute for $k=1,2$ the following norm
\begin{align*}
\norm{\tau_\xi\T_{\e}^{b,k}(v_{\e}^{k})-\T_{\e}^{b,k}(v_{\e}^{k})}_{L^{2}\left((0,T)\times \Omega^{\xi}_{\lambda} \times \Gamma^{k}\right)}^2
&=\norm{\tau_\xi\T_{\e}^{b,k}(v_{\e}^{k})-\T_{\e}^{b,k}(v_{\e}^{k})}_{L^{2}\left((0,T)\times (\Omega^{\xi}_{\lambda}\cap \widehat{\Omega}^{\e}) \times \Gamma^{k}\right)}^2
\\& \quad +\norm{\tau_h\T_{\e}^{b,k}(v_{\e}^{k})-\T_{\e}^{b,k}(v_{\e}^{k})}_{L^{2}\left((0,T)\times (\Omega^{\xi}_{\lambda}\setminus\widehat{\Omega}^{\e}) \times \Gamma^{k}\right)}^2
\\& := E_{a,\e}^{\xi}+E_{b,\e}^{\xi}.
\end{align*}
Proceeding in a similar way to \cite{soren,maria07}, we first estimate $E_{1,\e}^{\xi}$ using the above decomposition of the domain as follows:
\begin{align*}
E_{a,\e}^{\xi}& =\sum\limits_{h \in \Xi_\e}\int_{0}^{T} \int_{\e (h_{\ell} +Y)} \int_{\Gamma^{k}} \card{v_{\e}^{k}\left(t,\e\left[\dfrac{x+\xi}{\e} \right]_Y+\e y  \right)-v_{\e}^{k}\left(t,\e\left[\dfrac{x}{\e} \right]_Y+\e y  \right) }^2 d\sigma_ydxdt
\\ & =\sum\limits_{h \in \Xi_\e}\sum\limits_{\sigma\in\left\lbrace 0,1 \right\rbrace^d }\int_{0}^{T} \int_{\e (h_{\ell} +Y)^\sigma} \int_{\Gamma^{k}} \card{v_{\e}^{k}\left(t,\e\left( h_{\ell}+\sigma+\left[\dfrac{\xi}{\e} \right]_Y\right)+\e y  \right)-v_{\e}^{k}\left(t,\e h_{\ell}+\e y  \right) }^2 d\sigma_ydxdt
\\ & \leq\sum\limits_{h \in \Xi_\e}\sum\limits_{\sigma\in\left\lbrace 0,1 \right\rbrace^d }\int_{0}^{T} \int_{\e (h_{\ell} +Y)} \int_{\Gamma^{k}} \card{v_{\e}^{k}\left(t,\e\left( h_{\ell}+\sigma+\left[\dfrac{\xi}{\e} \right]_Y\right)+\e y  \right)-v_{\e}^{k}\left(t,\e h_{\ell}+\e y  \right) }^2 d\sigma_ydxdt
\\ & \leq \sum\limits_{\sigma\in\left\lbrace 0,1 \right\rbrace^d }\int_{0}^{T} \int_{\widehat{\Omega}^{\e}} \int_{\Gamma^{k}} \card{\T_{\e}^{b,k}v_{\e}^{k}\left(t,x+\e\left( \sigma+\left[\dfrac{\xi}{\e} \right]_Y\right), y  \right)-\T_{\e}^{b,k}v_{\e}^{k}\left(t, x, y \right) }^2 d\sigma_ydxdt,
\end{align*} 
which by using the integration formula \eqref{P_uo4_gap} $($for $p=2)$ of Proposition \ref{prop_uo_gap} is equal to
\begin{equation*}
\sum\limits_{\sigma\in\left\lbrace 0,1 \right\rbrace^d }\e  \abs{Y}\int_{0}^{T}  \int_{\Gamma_{\e}^{k}} \card{v_{\e}^{k}\left(t,x+\e\left( \sigma+\left[\dfrac{\xi}{\e} \right]_Y\right)  \right)-v_{\e}^{k}\left(t, x \right) }^2 d\sigma_ydt.
\end{equation*}

For a given small $\gamma>0,$ we can choose an $\e$ small enough such that $\card{\e \sigma+\e\left[\dfrac{\xi}{\e} \right]_Y}< \gamma.$ This amounts to saying that in order to estimate $E_{a,\e}^{\xi},$ it is sufficient to obtain estimates for given $\ell\in \mathbb{Z}^d,$ $\abs{\e \ell}<\gamma$ of 
\begin{equation}
\norm{v_{\e}^{k}\left(t,x+\e\ell\right)-v_{\e}^{k}\left(t, x \right)}_{L^{2}\left((0,T) \times \Gamma_{\e,Q}^{k}\right)}^2,
\label{E_trans_v_gap}
\end{equation}
where $\Gamma_{\e,Q}=\Gamma_{\e}\cap Q$ with $Q\subset \Omega$ an open set. 
 
 In order to estimate the norm \eqref{E_trans_v_gap}, we test the variational equation the weak formulation \eqref{Fv_ie_gap} with $\varphi_{i}^{k}=\eta^2\left( \tau_{\e\ell}u_{i,\e}^{\e}-u_{i,\e}^{k}\right) $ for $k=1,2$ and $\varphi_{e}=\eta^2\left( \tau_{\e\ell}u_{e,\e}-u_{e,\e}\right),$ where $\eta \in D(Q)$ is a cut-off function with $0\leq \eta \leq 1,$ $\eta=1 $ in $Q$ and zero outside a small neighborhood $Q'$ of $Q.$ Proceeding exactly as Lemma 5.2 in \cite{bendunf19}, Gronwall's inequality and the assumptions on the initial data give the following result:
 $$\e\norm{v_\e\left(t,x+\e\ell\right)-v_\e\left(t, x \right)}_{L^{2}\left((0,T) \times \Gamma_{\e,Q}\right)}^2\leq C\e\abs{\ell},$$
 where $C$ is a positive constant.
Then, we obtain by using the previous estimate
\begin{equation}
E_{a,\e}^{\xi}\leq C\left(\abs{\xi}+\e\right).
\label{E_ae_h}
\end{equation}
Hence, we can deduce that $E_{a,\e}^{h}\rightarrow 0$ as $\xi\rightarrow 0$ uniformly in $\e$, as in \cite{mariahom}. Indeed, to prove that
\begin{equation}
\forall \rho>0, \exists \mu>0 \text{ such that for every } \e \text{ tends to } 0^{+}, \ \forall \xi, \ \abs{\xi}\leq \mu \Rightarrow E_{a,\e}^{\xi}<\rho,
\label{E_a}
\end{equation}
one identifies two cases:
\begin{itemize}
\item[$(a)$] For $0<\e<\dfrac{\rho}{2C}:$ take $\mu=\dfrac{\rho}{2C},$ then, from \eqref{E_ae_h}, we get that condition \eqref{E_a} holds for $\abs{\xi}\leq \mu.$ 
\item[$(b)$] For $\dfrac{\rho}{2C}<\e < 1:$  we remark that since $\e$ tends to $0^{+}$,  there are only finitely many elements $\e$ in the interval $(\frac{\rho}{2C},1),$ say $\lbrace \e_n \rbrace_{n=1}^{N}$ with $N\in \mathbb{N},$ $N<\infty.$ Moreover, by the continuity of translations in the mean of $L^2$-functions, for every $n,$ $\exists \mu_n=\mu(\e_n)$ such that $\forall \xi, \ \abs{\xi}\leq \mu_n,$ condition \eqref{E_a} holds. Thus choosing $\mu=\min\lbrace\frac{\rho}{2C}, \mu_1, \dots, \mu_N\rbrace$ together with the argument for the translation with respect to time, property \eqref{E_a} is proved. 
\end{itemize}

It easy to check that $$E_{b,\e}^{\xi}=\norm{\tau_\xi\T_{\e}^{b,k}(v_{\e}^{k})}_{L^{2}\left((0,T)\times (\Omega^{\xi}_{\lambda}\setminus\widehat{\Omega}^{\e}) \times \Gamma^{k}\right)}^2\leq \norm{\tau_\xi\T_{\e}^{b,k}(v_{\e}^{k})}_{L^{2}\left((0,T)\times (\Omega_{\lambda}\setminus\widehat{\Omega}^{\e}) \times \Gamma^{k}\right)}^2.$$
Hence, we can deduce that $E_{b,\e}^{\xi}\rightarrow 0$ as $\xi\rightarrow 0$ uniformly in $\e.$ Indeed, to prove that
\begin{equation}
\forall \rho>0, \exists \mu>0 \text{ such that } \forall \e>0, \ \forall \xi, \ \abs{\xi}\leq \mu \Rightarrow E_{b,\e}^{\xi}<\rho,
\label{E_b}
\end{equation}
one identifies two cases:
\begin{itemize}
\item[$(a)$] For $\e$ small enough, say $\e<\e_0,$ $\Omega_{\lambda}\subset\widehat{\Omega}^{\e},$ then $E_{b,\e}^{\xi}=0.$
\item[$(b)$] For $\e_0<\e < 1:$  we remark that since $\e$ tends to $0^{+}$,  there are only finitely many elements $\e$ in the interval $(\e_0,1),$ say $\lbrace \e_n \rbrace_{n=1}^{N}$ with $N\in \mathbb{N},$ $N<\infty.$ Moreover, by the continuity of translations in the mean of $L^2$-functions, for every $n,$ $\exists \mu_n=\mu(\e_n)$ such that $\forall \xi, \ \abs{\xi}\leq \mu_n,$ condition \eqref{E_b} holds. Thus choosing $\mu=\min\lbrace\frac{\rho}{2C}, \mu_1, \dots, \mu_N\rbrace$ together with the argument for the translation with respect to time, property \eqref{E_b} is proved.
\end{itemize}
This ends the proof of the condition (ii) in Proposition \ref{kolmo_gap}.

$\textbf{(iii)}$ The last condition follows from the a priori estimate \eqref{E_vr_gap}. Indeed, we have for $k=1,2$:
\begin{equation*}
\int_{0}^{T}\int_{\Omega\setminus\Omega_{\lambda}}\card{\T_{\e}^{b,k}(v_{\e}^{k})}^2 dxdt \leq \abs{\Omega\setminus\Omega_{\lambda}}^{\frac{r-2}{r}}\left(\int_{\Omega_T}\card{\T_{\e}^{b,k}(v_{\e}^{k})}^r dxdt \right)^{\frac{2}{r}}\leq C \abs{\Omega\setminus\Omega_{\lambda}}^{\frac{r-2}{r}}.
\end{equation*}
The conditions (i)-(iii) imply that the Kolmogorov criterion for $\T_{\e}^{b,k}(v_\e)$ holds true in $L^{2}(\Omega_T\times \Gamma^{k})$ for $k=1,2.$ This concludes the proof of the first convergence in our Lemma.

 It remains to prove the second convergence which will be done as follows. Note that from the structure of $\mathrm{I}_{a,ion}$ and using property \eqref{P_uo2_gap} in Proposition \ref{prop_uo_gap}, we have
\begin{equation*}
\T_{\e}^{b,k}\left(\mathrm{I}_{a,ion}(v_{\e}^{k})\right)=\mathrm{I}_{a,ion}\left( \T_{\e}^{b,k}(v_{\e}^{k})\right), \text{ for  } k=1,2.
\end{equation*}
Due to the strong convergence of $\T_{\e}^{b,k}(v_{\e}^{k})$ in $L^2(\Omega_T\times \Gamma^{y}),$ we can extract a subsequence, such that $\T_{\e}^{b,k}(v_{\e}^{k})\rightarrow v^{k}$ a.e. in $\Omega_T\times \Gamma^{k}$ with $k=1,2.$ Since $\mathrm{I}_{a,ion}$ is continuous, we have
\begin{equation*}
\mathrm{I}_{a,ion}\left( \T_{\e}^{b,k}(v_{\e}^{k})\right)\rightarrow \mathrm{I}_{1,ion}(v^{k}) \text{ a.e. in } \Omega_T\times \Gamma^{y}.
\end{equation*}
Further, we use estimate \eqref{E_vr_gap} with property \eqref{P_uo4_gap} of Proposition \ref{prop_uo_gap} to obtain for $k=1,2$ 
$$ \norm{\T_{\e}^{b,k}\left(\mathrm{I}_{a,ion}(v_{\e}^{k})\right)}_{L^{r/(r-1)}\left(\Omega_T \times \Gamma^{y}\right)}\leq\abs{Y}^{(r-1)/r}\norm{\e^{(r-1)/r}\mathrm{I}_{a,ion}(v_{\e}^{k})}_{L^{r/(r-1)}(\Gamma_{\e,T})}\leq C.$$
Hence, using a classical result (see Lemma 1.3 in \cite{lions1969}):
\begin{equation*}
\mathrm{I}_{a,ion}\left( \T_{\e}^{b,k}(v_{\e}^{k})\right)\rightharpoonup\mathrm{I}_{a,ion}(v^{k})\text{ weakly in } L^{r/(r-1)}(\Omega_T\times \Gamma^{k}) \text{ with } k=1,2.
\end{equation*}
Moreover, we obtain, using Vitali's Theorem, the strong convergence of $\mathrm{I}_{a,ion}\left( \T_{\e}^{b,k}(v_{\e}^{k})\right)$ to $\mathrm{I}_{a,ion}(v^{k})$ in $L^q(\Omega_T\times \Gamma^{k}), \ \forall q\in[1,r/(r-1))$ and $k=1,2.$ This finishes the proof of our Lemma.
\end{proof}
Finally, we pass to the limit when $\e\rightarrow 0$ in the unfolded formulation \eqref{Fv_ie_unf_gap} to obtain the following limiting problem:
\begin{equation}
\begin{aligned}
&\dfrac{1}{\abs{Y}}\sum \limits_{k=1,2} \iiint_{\Omega_{T}\times \Gamma^{k}} \pt_t v^{k} \psi^{k} \ dxd\sigma_ydt+\dfrac{1}{2\abs{Y}}\iiint_{\Omega_{T}\times \Gamma^{1,2}} \pt_t s \Psi \ dxd\sigma_ydt
\\& \quad +\dfrac{1}{\abs{Y}}\sum \limits_{k=1,2}\iiint_{\Omega_{T}\times Y_{i}^{k}}\mathrm{M}_{i} \left[ \nabla u_{i}^{k}+\nabla_y \widehat{u}_{i}^{k}\right] \nabla \varphi_{i}^{k} \ dxdydt
\\&\quad +\dfrac{1}{\abs{Y}}\iiint_{\Omega_{T}\times Y_{e}}\mathrm{M}_{e} \left[ \nabla u_{e}+\nabla_y \widehat{u}_{e}\right] \nabla \varphi_e \ dxdydt
\\& \quad +\dfrac{1}{\abs{Y}}\sum \limits_{k=1,2}\iiint_{\Omega_{T}\times \Gamma^{k}}\mathrm{I}_{a,ion}(v^{k}) \psi^{k} \ dxd\sigma_ydt+\dfrac{1}{\abs{Y}}\sum \limits_{k=1,2}\iiint_{\Omega_{T}\times \Gamma^{k}}\mathrm{I}_{b,ion}(w^{k}) \psi^{k} \ dxd\sigma_ydt
\\& \quad +\dfrac{1}{2\abs{Y}}\iint_{\Omega_{T}\times \Gamma^{1,2}}\I_{gap}(s) \Psi \ dxd\sigma_ydt
\\&=\dfrac{1}{\abs{Y}}\sum \limits_{k=1,2}\iiint_{\Omega_{T}\times \Gamma^{k}} \I_{app}^{k}\psi^{k} \ dxd\sigma_ydt,
\end{aligned}
\label{Fv_ie_hom}
\end{equation}
Similarly, we can prove also that the limit of \eqref{Fv_d_unf_gap} for $k=1,2$ as $\e$ tends to zero, is given by:
\begin{equation}
\dfrac{1}{\abs{Y}}\iiint_{\Omega_{T}\times \Gamma^{k}} \pt_t w^{k} e^{k} \ dxd\sigma_ydt-\dfrac{1}{\abs{Y}}\iiint_{\Omega_{T}\times \Gamma^{k}}  H(v^{k}, w^{k}) e^{k} \ dxd\sigma_ydt =0.
\label{Fv_d_hom}
\end{equation}

\begin{rem}
Since the linear term $H$ is not varying at the micro scale and since $v^{k}$ does not depend on $y$, it can be proven, using Assumption \eqref{A_H_Ib_a}, that the solution $w^{k}$ of 
\begin{equation*}
\begin{cases}
\pt_t w^{k}=H(v^{k}, w^{k}) &\text{ in } \Omega_{T}\times \Gamma^{k},
\\  w^{k}(0,x)=w_{0}^{k}(x) &\text{ on } \Omega,
\end{cases}
\end{equation*}
 is unique for all  $y\in \Gamma^{k}$ for $k=1,2$ hence it is independent of the variable $y$.
\end{rem}

\subsection{Derivation of the macroscopic tridomain model}\label{macro_gap}
The convergence results of the previous part allow us to pass to the limit in the microscopic equations \eqref{Fv_ik_ini_gap}-\eqref{Fv_d_ini_gap}  and to obtain the homogenized model formulated in Theorem \ref{thm_macro_gap}.

  To this end, we choose a special form of test functions to capture the  microscopic informations at each structural level. Then, we consider that the test functions have the following form:
\begin{equation}
\begin{cases}
 \varphi_{e,\e}=\phi_{e}(t,x)+\e\theta_{e}(t,x)\Theta_{e,\e}(x),
\\ \varphi_{i,\e}^{k}=\phi_{i}^{k}(t,x)+\e\theta_{i}^{k}(t,x)\Theta_{i,\e}^{k}(x),
\end{cases}
\end{equation}
with functions $\Theta_{e,\e}$ and $\Theta_{i,\e}^{k}$ for $k=1,2$ defined by: 
$$\Theta_{e,\e}(x)=\Theta_{e}\left( \dfrac{x}{\e}\right)\ \text{ and } \Theta_{i,\e}^{k}(x)=\Theta_{i}^{k}\left( \dfrac{x}{\e}\right), \ \text{ for } k=1,2,$$
where $\phi_{e}, \phi_{i}^{k}, \theta_{e}$ and $\theta_{i}^{k}$ are in $D(\Omega_T),$ $\Theta_{e}$ in $H_{\#}^1(Y_{e})$ and $\Theta_{i}^{k}$ in $H_{\#}^1(Y_{i}^{k})$ for $k=1,2.$
 Then, we have:
 \begin{equation*}
 \begin{cases}
 \nabla \varphi_{e,\e}=\nabla_x \phi_{e}+\e\nabla_x\theta_{e} \Theta_{e,\e}+\theta_{e} \nabla_y\Theta_{e,\e},
 \\ \nabla \varphi_{i,\e}^{k}=\nabla_x \phi_{i}^{k}+\e\nabla_x\theta_{i}^{k} \Theta_{i,\e}^{k}+\theta_{i}^{k} \nabla_y\Theta_{i,\e}^{k}.
 \end{cases}
 \end{equation*}
 Due to the regularity of test functions and using property \eqref{P_uo5_gap} of Proposition \ref{prop_uo_gap}, there holds for $k=1,2$, when $\e \rightarrow 0:$ 
 \begin{align*}
&\T_{\e}^{i,k}(\varphi_{i,\e}^{k})\rightarrow \phi_{i}^{k} \text{ strongly in } L^{2}\left(\Omega_T \times Y_{i}^{k}\right),
\\& \T_{\e}^{i,k}(\theta_{i}^{k}\Theta_{i,\e}^{k})\rightarrow \theta_{i}^{k}(t,x)\Theta_{i}^{k}(y)\text{ strongly in } L^{2}\left(\Omega_T \times Y_{i}^{k}\right),
\\& \T_{\e}^{i,k}\left( \nabla \varphi_{i,\e}^{k}\right)\rightarrow \nabla_x \phi_{i}^{k}+\theta_{i}^{k} \nabla_y\Theta_{i,\e}^{k} \text{ strongly in } L^{2}\left(\Omega_T \times Y_{i}^{k}\right),
\\&\T_{\e}^{e}(\varphi_{e,\e})\rightarrow \phi_{e} \text{ strongly in } L^{2}\left(\Omega_T \times Y_{e}\right),
\\& \T_{\e}^{e}(\theta_{e}\Theta_{e,\e})\rightarrow \theta_{e}(t,x)\Theta_{e}(y)\text{ strongly in } L^{2}\left(\Omega_T \times Y_{e}\right),
\\& \T_{\e}^{e}\left( \nabla \varphi_{e,\e}\right)\rightarrow \nabla_x \phi_{e}+\theta_{e} \nabla_y\Theta_{e,\e} \text{ strongly in } L^{2}\left(\Omega_T \times Y_{e}\right).
\end{align*}
Since $\psi_{\e}^{k}:=\left(\varphi_{i,\e}^{k}-\varphi_{e,\e} \right)\vert_{\Gamma_{\e,T}^{k}} $ for $k=1,2$ and  $\Psi_{\e}:=\left(\varphi_{i,\e}^{1}-\varphi_{i,\e}^{2} \right)\vert_{\Gamma_{\e,T}^{1,2}},$ then it holds also: 
 \begin{align*}
 & \T_{\e}^{b,k}(\psi_{\e}^{k}) \rightarrow \psi^{k} \text{ strongly in } L^{2}(\Omega_T\times \Gamma^{k}),
\\& \T_{\e}^{b,1,2}(\Psi_{\e}) \rightarrow \Psi \text{ strongly in } L^{2}(\Omega_T\times \Gamma^{1,2}),
\end{align*}
where $\psi^{k}:=\left(\phi_{i}^{k}-\phi_{e} \right)\vert_{\Omega_T\times \Gamma^{k}}$ for $k=1,2$ and  $\Psi:=\left(\phi_{i}^{1}-\phi_{i}^{2} \right)\vert_{\Omega_T\times\Gamma^{1,2}}.$
  
   Collecting all the convergence results of $J_1,\dots,J_8$ obtained in Section \ref{unf_gap}, we deduce the following limiting problem:
   \begin{equation}
\begin{aligned}
&\sum \limits_{k=1,2} \dfrac{\abs{\Gamma^{k}}}{\abs{Y}}\iint_{\Omega_{T}} \pt_t v^{k} \psi^{k} \ dxdt+\dfrac{\abs{\Gamma^{1,2}}}{2\abs{Y}}\iint_{\Omega_{T}} \pt_t s \Psi \ dxdt
\\& \quad +\dfrac{1}{\abs{Y}}\sum \limits_{k=1,2}\iiint_{\Omega_{T}\times Y_{i}^{k}}\mathrm{M}_{i} \left[ \nabla u_{i}^{k}+\nabla_y \widehat{u}_{i}^{k}\right] \left[\nabla_x \phi_{i}^{k}+\theta_{i}^{k} \nabla_y\Theta_{i,\e}^{k}\right] \ dxdydt
\\& \quad +\dfrac{1}{\abs{Y}}\iiint_{\Omega_{T}\times Y_{e}}\mathrm{M}_{e} \left[ \nabla u_{e}+\nabla_y \widehat{u}_{e}\right] \left[\nabla_x \phi_{e}+\theta_{e} \nabla_y\Theta_{e,\e}\right]  \ dxdydt
\\& \quad +\sum \limits_{k=1,2}\dfrac{\abs{\Gamma^{k}}}{\abs{Y}}\iint_{\Omega_{T}}\mathrm{I}_{a,ion}(v^{k}) \psi^{k} \ dxdt+\sum \limits_{k=1,2}\dfrac{\abs{\Gamma^{k}}}{\abs{Y}}\iint_{\Omega_{T}}\mathrm{I}_{b,ion}(w^{k}) \psi^{k} \ dxdt
\\& \quad +\dfrac{\abs{\Gamma^{1,2}}}{2\abs{Y}}\iint_{\Omega_{T}}\I_{gap}(s) \Psi \ dxdt
\\&=\sum \limits_{k=1,2}\dfrac{\abs{\Gamma^{k}}}{\abs{Y}}\iint_{\Omega_{T}} \I_{app}^{k}\psi^{k} \ dxdt.
\end{aligned}
\label{Fv_ie_psi_theta_gap}
\end{equation}
Similarly, we can prove also that the limit of the coupled dynamic equation for $k=1,2$ as $\e$ tends to zero, which is given by:
\begin{equation}
\dfrac{\abs{\Gamma^{k}}}{\abs{Y}}\iint_{\Omega_{T}} \pt_t w e^{k} \ dxdt-\dfrac{\abs{\Gamma^{k}}}{\abs{Y}}\iint_{\Omega_{T}}  H(v^{k}, w^{k}) e^{k} \ dxdt =0.
\label{Fv_d_psi_theta_gap}
\end{equation}

 Now, we will find first the expression of $\widehat{u}_{i}^{k}$ in terms of the homogenized solution $u_{i}^{k}$ for $k=1,2.$ Then, we derive the cell problem from the homogenized equation \eqref{Fv_ie_psi_theta_gap}. Finally, we obtain the weak formulation of the corresponding macroscopic equation.
 
   We first take $\phi_{e},$ $\theta_{e}$ and $\phi_{i}^{k}$ for $k=1,2$ are equal to zero, to get:
\begin{equation}
\dfrac{1}{\abs{Y}}\sum \limits_{k=1,2}\iint_{\Omega_{T}\times Y_{i}^{k}}\mathrm{M}_{i} \left( \nabla u_{i}^{k}+\nabla_y \widehat{u}_{i}^{k}\right) \left(\theta_{i}^{k} \nabla_y\Theta_{i,\e}^{k}\right) \ dxdydt=0.
\label{Fv_i_theta1_theta2_gap}
\end{equation} 
Since $u_{i}^{k}, k=1,2$ is independent on the microscopic variable $y$ then the formulation \eqref{Fv_i_theta1_theta2_gap} corresponds to the following microscopic problem:
\begin{equation}
 \begin{cases}
-\nabla_y\cdot\left(\mathrm{M}_i \nabla_y \widehat{u}_{i}^{k}\right) =\overset{d}{\underset{p,q=1}{\sum}}\dfrac{\pt \mathrm{m}^{pq}_{i}}{\pt y_p}\dfrac{\pt u_{i}^{k}}{\pt x_q} \ \text{in} \ Y_{i}^{k},
\\ \widehat{u}_{i}^{k} \ y\text{-periodic}, \\ \left( \mathrm{M}_{i}\nabla_y \widehat{u}_{i}^{k}+ \mathrm{M}_{i}\nabla_x u_{i}^{k}\right)  \cdot n_{i}^{k}= 0 \ \text{on} \ \Gamma^{k}, \\ \left( \mathrm{M}_{i}\nabla_y \widehat{u}_{i}^{k}+ \mathrm{M}_{i}\nabla_x u_{i}^{k}\right)  \cdot n_{i}^{k}= 0 \ \text{on} \ \Gamma^{1,2}. 
 \end{cases}
 \label{Ayyhatu_i_gap}
 \end{equation}
Hence, by the $y$-periodcity of $\mathrm{M}_i$ and the compatibility condition, it is not difficult to establish the existence of a unique periodic solution up to an additive constant of the problem \eqref{Ayyhatu_i_gap} (see \cite{BaderDev} for more details).\\
 Thus, the linearity of terms in the right-hand side of \eqref{Ayyhatu_i_gap} suggests to look for $\widehat{u}_{i}^{k}$ under the following form in terms of $u_{i}^{k}$:
\begin{equation} 
\widehat{u}_{i}^{k}(t,x,y,z)=\chi_i(y)\cdot \nabla_x u_{i}^{k}+\widehat{u}_{0,i}^{k}(t,x,y),
\label{hatu_i_gap}
\end{equation}
where $\widehat{u}_{0,i}^{k}, k=1,2$ is a constant with respect to $y$ and each element $\chi_i^q$ of $\chi_i$ satisfies the following $\e$-cell problem:
\begin{equation}
\begin{cases}
-\nabla_y \cdot\left(\mathrm{M}_i\nabla_y \chi_i^q \right)=\overset{d}{\underset{p=1}{\sum}}\dfrac{\pt \mathrm{m}^{pq}_{i}}{\pt y_p} \ \text{in} \ Y_{i}^{k},\\ \chi_i^q \ y\text{-periodic}, \\ \mathrm{M}_{i} \nabla_y \chi_i^q \cdot n_{i}^{k}= - (\mathrm{M}_{i}e_q )\cdot n_{i}^{k} \text{ on } \Gamma^{k}, \ k=1,2 
\\ \mathrm{M}_{i} \nabla_y \chi_i^q \cdot n_{i}^{k}= - (\mathrm{M}_{i}e_q )\cdot n_{i}^{k} \text{ on } \Gamma^{1,2},
 \end{cases}
 \label{Ayychiik_gap}
 \end{equation}
 for $q=1,\dots,d.$ Moreover, the compatibility  condition is imposed to guarantee the existence and uniqueness of solution $\chi_{i}^q \in H_{\#}^1(Y_{i}^{k})$ to problem \eqref{Ayychiik_gap} with $H_{\#}^1$ is given by \eqref{W_gap}. 

 Finally, inserting the form \eqref{hatu_i_gap} of $\widehat{u}_{i}^{k}$ into \eqref{Fv_ie_psi_theta_gap} and setting $\theta_{i}^{k}, \theta_{e}$ $\phi_{e}$ to zero, one obtains the weak formulation of the homogenized equation for the intracellular problem:
  \begin{equation}
\begin{aligned}
& \sum \limits_{k=1,2}\mu_{k}\iint_{\Omega_{T}} \pt_t v^{k} \phi_{i}^{k} \ dxdt+\mu_{g}\iint_{\Omega_{T}} \pt_t s \phi_{i}^{1} \ dxdt
\\&+\dfrac{1}{\abs{Y}}\sum \limits_{k=1,2}\iiint_{\Omega_{T}\times Y_{i}^{k}}\widetilde{\mathbf{M}}_i\nabla u_{i}^{k} \cdot\nabla \phi_{i}^{k} \ dxdydt
\\&+\sum \limits_{k=1,2}\mu_{k}\iint_{\Omega_{T}}\mathrm{I}_{a,ion}(v^{k}) \phi_{i}^{k} \ dxdt+\sum \limits_{k=1,2}\mu_{k}\iint_{\Omega_{T}}\mathrm{I}_{b,ion}(w^{k}) \phi_{i}^{k} \ dxdt
\\&+\mu_{g}\iint_{\Omega_{T}}\I_{gap}(s) \phi_{i}^{1} \ dxdt =\sum \limits_{k=1,2}\mu_{k}\iint_{\Omega_{T}} \I_{app}^{k}\phi_{i}^{k} \ dxdt,
\end{aligned}
\label{Fvhomik_gap}
\end{equation}
with  $\mu_{k}=\abs{\Gamma^{k}}/\abs{Y},$ $k=1,2,$ $\mu_{g}=\abs{\Gamma^{1,2}}/2\abs{Y}$ and the coefficients of the homogenized conductivity matrices $\widetilde{\mathbf{M}}_i=\left( \widetilde{\mathbf{m}}^{pq}_i\right)_{1\leq p,q \leq d}$ defined by:
\begin{equation}
\widetilde{\mathbf{m}}_{i}^{pq}:=\dfrac{1}{\abs{Y}}\overset{d}{\underset{\ell=1}{\sum}}\displaystyle\int_{Y_{i}^{k}}\left( \mathrm{m}_{i}^{pq}+\mathrm{m}^{p\ell}_{i}\dfrac{\pt \chi_i^q}{\pt y_\ell}\right) \ dy.
\label{Mt_ik_gap}
\end{equation} 
 Similarly, we can decouple the cell problem in the extracellular domain and define the homogenized matrix $\widetilde{\mathbf{M}}_e.$ This completes the proof of Theorem \ref{thm_macro_gap} using periodic unfolding method.
 \begin{rem}\label{Mt_pos_ell} $ $
 \begin{enumerate}
 \item[$1.$] Since the conductivity matrices $\mathrm{M}_j$ for $j=i,e$ are symmetric then the homogenized conductivity matrices  $\widetilde{\mathbf{M}}_j$ defined by \eqref{tilde_m_i_gap}-\eqref{tilde_m_e_gap} are also symmetric for $j=i,e.$
 \item[$2.$] We can rewrite the homogenized conductivity matrices $\widetilde{\mathbf{M}}_i=\left( \widetilde{\mathbf{m}}^{pq}_i\right)_{1\leq p,q \leq d}$ as follows
\begin{equation}
\widetilde{\mathbf{m}}_{i}^{pq}:=\dfrac{1}{\abs{Y}}\overset{d}{\underset{\ell,\ell'=1}{\sum}}\displaystyle\int_{Y_{i}^{k}} \mathrm{m}_{i}^{\ell\ell'}\dfrac{\pt \left(y_{q}+\chi_i^q\right)}{\pt y_{\ell'}}\dfrac{\pt \left(y_{p}+\chi_i^p\right)}{\pt y_{\ell}}\ dy.
\label{Mt_ik_gap_new1}
\end{equation}
Indeed, we recall that $\chi_i^q$ is the solution of \eqref{Ayychiik_gap}. Choosing $\chi_i^p$ as test function in \eqref{Ayychiik_gap}, one has
\begin{equation*}
\overset{d}{\underset{\ell,\ell'=1}{\sum}}\displaystyle\int_{Y_{i}^{k}} \mathrm{m}_{i}^{\ell\ell'}\dfrac{\pt \chi_i^q}{\pt y_{\ell'}}\dfrac{\pt\chi_i^p}{\pt y_{\ell}}\ dy=-\overset{d}{\underset{\ell=1}{\sum}}\displaystyle\int_{Y_{i}^{k}} \mathrm{m}_{i}^{\ell q}\dfrac{\pt\chi_i^p}{\pt y_{\ell}}\ dy=-\overset{d}{\underset{\ell,\ell'=1}{\sum}}\displaystyle\int_{Y_{i}^{k}} \mathrm{m}_{i}^{\ell\ell'}\dfrac{\pt y_{q}}{\pt y_{\ell'}}\dfrac{\pt \chi_i^p}{\pt y_{\ell}}\ dy.
\end{equation*}
Hence, one obtains
\begin{equation}
\dfrac{1}{\abs{Y}}\overset{d}{\underset{\ell,\ell'=1}{\sum}}\displaystyle\int_{Y_{i}^{k}} \mathrm{m}_{i}^{\ell\ell'}\dfrac{\pt \left(y_{q}+\chi_i^q\right)}{\pt y_{\ell'}}\dfrac{\pt \chi_i^p}{\pt y_{\ell}}\ dy=0.
\label{chi_i^q_chi_i^p_gap}
\end{equation}
On the other hand, since
\begin{align*}
&\int_{Y_{i}^{k}}\mathrm{m}_{i}^{pq} \ dy=\overset{d}{\underset{\ell,\ell'=1}{\sum}}\displaystyle\int_{Y_{i}^{k}} \mathrm{m}_{i}^{\ell\ell'}\dfrac{\pt y_{q}}{\pt y_{\ell'}}\dfrac{\pt y_{p}}{\pt y_{\ell}}\ dy,
\\& \overset{d}{\underset{\ell=1}{\sum}}\displaystyle\int_{Y_{i}^{k}}\mathrm{m}^{p\ell}_{i}\dfrac{\pt \chi_i^q}{\pt y_\ell} \ dy=\overset{d}{\underset{\ell,\ell'=1}{\sum}}\displaystyle\int_{Y_{i}^{k}} \mathrm{m}_{i}^{\ell\ell'}\dfrac{\pt \chi_i^q}{\pt y_{\ell'}}\dfrac{\pt y_{p}}{\pt y_{\ell}}\ dy,
\end{align*}
formula \eqref{Mt_ik_gap} can be written as follows:
\begin{equation}
\widetilde{\mathbf{m}}_{i}^{pq}=\dfrac{1}{\abs{Y}}\overset{d}{\underset{\ell,\ell'=1}{\sum}}\displaystyle\int_{Y_{i}^{k}} \mathrm{m}_{i}^{\ell\ell'}\dfrac{\pt \left(y_{q}+\chi_i^q\right)}{\pt y_{\ell'}}\dfrac{\pt y_{p}}{\pt y_{\ell}}\ dy, \ \forall p,q=1,\dots,d.
\label{Mt_ik_gap_new2}
\end{equation}
Summing \eqref{chi_i^q_chi_i^p_gap} from \eqref{Mt_ik_gap_new2} gives \eqref{Mt_ik_gap_new1}. Similarly, we can rewrite the other matrix $\widetilde{\mathbf{M}}_e$ in terms of the corresponding corrector function $\chi_e.$ 
 \item[$3.$] Since the conductivity matrices $\mathrm{M}_j$ for $j=i,e$ satisfy the elliptic conditions defined by \eqref{A_Mie_gap}, then the homogenized conductivity matrices $\widetilde{\mathbf{M}}_j,$ $j=i,e$ verify the following elliptic conditions: there exits $\alpha_0,\beta_0 >0$ such that
\begin{subequations}
\begin{align}
&\widetilde{\mathbf{M}}_j\lambda\cdot\lambda \geq \alpha_0\abs{\lambda}^2,
\label{A_Mtie_elli_gap}
\\& \abs{\widetilde{\mathbf{M}}_j\lambda}\leq \beta_0 \abs{\lambda}, \ \text{ for any }  \lambda\in \R^d.
\label{A_Mtie_con_gap}
\end{align}
\label{A_Mtie_gap}
\end{subequations}

Indeed, let $\lambda \in \R^d$ and $j=i.$ To prove \eqref{A_Mtie_elli_gap}, then from \eqref{Mt_ik_gap_new1} it follows that
 \begin{equation*}
 \overset{d}{\underset{p,q=1}{\sum}}\widetilde{\mathbf{m}}_{i}^{pq}\lambda_p\lambda_q=\dfrac{1}{\abs{Y}}\overset{d}{\underset{p,q=1}{\sum}}\overset{d}{\underset{\ell,\ell'=1}{\sum}}\displaystyle\int_{Y_{i}^{k}} \mathrm{m}_{i}^{\ell\ell'}\lambda_p\dfrac{\pt \left(y_{p}+\chi_i^p\right)}{\pt y_{\ell}}\lambda_q\dfrac{\pt \left(y_{q}+\chi_i^q\right)}{\pt y_{\ell'}}\ dy.
 \end{equation*}
 Setting $\zeta_i=\overset{d}{\underset{p=1}{\sum}}\lambda_p\left(y_{p}+\chi_i^p\right)$ and using the ellipticity of $\mathrm{M}_i$ defined by \eqref{A_Mie_gap}, we get
 \begin{equation}
 \overset{d}{\underset{p,q=1}{\sum}}\widetilde{\mathbf{m}}_{i}^{pq}\lambda_p\lambda_q \geq \dfrac{\alpha}{\abs{Y}}\displaystyle\int_{Y_{i}^{k}} \card{\nabla \zeta_i}^2\ dy \geq 0,  \text{ for any } \lambda \in \R^d.
 \label{Mt_ie_zeta_gap}
 \end{equation}
 Let us show that this inequality implies that
  \begin{equation*}
 \overset{d}{\underset{p,q=1}{\sum}}\widetilde{\mathbf{m}}_{i}^{pq}\lambda_p\lambda_q > 0,  \text{ for any } \lambda \in \R^d, \ \lambda\neq 0.
 \end{equation*}
 If this were not true. In view of \eqref{Mt_ie_zeta_gap}, one would have some  $\lambda\neq 0$ such that
$$\card{\nabla \zeta_i}=0. $$
This means that
$$ \zeta_i=\overset{d}{\underset{p=1}{\sum}}\lambda_p\left(y_{p}+\chi_i^p\right)=constant.$$
Thus, one has 
$$\overset{d}{\underset{p=1}{\sum}}\lambda_p y_{p}=-\overset{d}{\underset{p=1}{\sum}}\lambda_p \chi_i^p + C, $$
and this impossible since the right-hand side function is $y$-periodic by definition and $\lambda\neq 0.$ To end the proof of ellipticity, we know that the function $\overset{d}{\underset{p,q=1}{\sum}}\widetilde{\mathbf{m}}_{i}^{pq}\xi_p\xi_q$ is continuous on the unit sphere $\mathbb{S}^{d-1}$ which is a compact set of $\R^d.$ Hence, this function achieves its minimum on $\mathbb{S}^{d-1}$ and, due to the previous result, this minimum is positive. So, there exists $\alpha_0>0$ such that 
$$\overset{d}{\underset{p,q=1}{\sum}}\widetilde{\mathbf{m}}_{i}^{pq}\xi_p\xi_q \geq \alpha_0, \ \forall \xi\in \mathbb{S}^{d-1}.$$ Consequently,
$$ \overset{d}{\underset{p,q=1}{\sum}}\widetilde{\mathbf{m}}_{i}^{pq}\dfrac{\lambda_p}{\abs{\lambda}}\dfrac{\lambda_q}{\abs{\lambda}} \geq \alpha_0,  \text{ for any } \lambda \in \R^d, \ \lambda \neq 0,$$ 
since the vector $\left(\dfrac{\lambda_1}{\abs{\lambda}},\dots,\dfrac{\lambda_d}{\abs{\lambda}} \right)$ belongs to $\mathbb{S}^{d-1}.$ This ends the proof of inequality $\eqref{A_Mtie_elli_gap}$ and by the same way we obtain the second inequality.
\end{enumerate}  

 \end{rem}


\bibliographystyle{plain}
 \bibliography{Hom}
\appendix

\end{document}